\documentclass[a4paper,11pt]{article}
\usepackage[utf8]{inputenc}
\usepackage{setspace}  
\usepackage{hyperref}

\usepackage{authblk}  
\usepackage{bbm}
\usepackage{mathtools}
\usepackage{commath}
\usepackage{amsmath}
\usepackage{amsthm}
\usepackage{amssymb}
\usepackage{amsfonts}
\usepackage{graphicx}
\graphicspath{ {./images/} } 
\usepackage{color}
\usepackage{calc}
\usepackage{IEEEtrantools}
\usepackage[percent]{overpic}
\usepackage{xcolor}
\usepackage{float}
\usepackage{enumerate}
\usepackage{tikz-cd}
\usetikzlibrary{decorations.pathmorphing}
\tikzset{snake it/.style={decorate, decoration=snake}}
\usepackage{lmodern}
\usetikzlibrary{arrows}
\def\E{\mathbb{E}}

\def\P{\mathbb{P}}
\def\pr{\mathbb{P}}

\newcommand{\bracedincludegraphics}[2][]{%
\newcommand{\ind}[1]{{\rm \bf1}\{#1\}}

  \sbox0{$\vcenter{\hbox{\includegraphics[#1]{#2}}}$}%
  \left\lbrace
    \vphantom{\copy0}
  \right.\kern-\nulldelimiterspace
  \underbrace{\vrule width0pt depth \dimexpr\dp0 + .3ex\relax\box0}}

\newenvironment{assump}[2][]
{\renewcommand\theassumption{#2}\begin{assumption}[#1]}
    {\end{assumption}}

\numberwithin{equation}{section} % Number equations separately

\newtheorem{theorem}{Theorem}
\newtheorem{lemma}[theorem]{Lemma}

\newtheorem{assumption}[theorem]{Assumption}

\theoremstyle{definition}

\newtheorem{remark}[theorem]{Remark}

\title{Random walks in Weyl chambers}
\author{Denis Denisov, Will FitzGerald and Kaiyuan Zhang
\thanks{
  {\it Emails:} denis.denisov@manchester.ac.uk, william.fitzgerald@manchester.ac.uk 
  and kaiyuan.zhang@postgrad.manchester.ac.uk.
   \newline
    The research of D. Denisov and W. FitzGerald was supported by a Leverhulme Trust Research Project Grant  RPG-2021-105.\newline  
    {\it AMS 2020 subject classifications: \/}
    Primary 60G50, 60G40; secondary  60C05, 60F17.\newline 
    {\it Key words and phrases.\/} Ordered random walks, Doob h-transform, Weyl chamber, discrete harmonic function. 
}
}
\affil{University of Manchester}

\begin{document}
\maketitle

\begin{abstract}
We study a $d$-dimensional random walk with zero mean and finite variance in the Weyl chambers of type C and D. Under optimal moment assumptions we construct  positive harmonic functions for  random walks killed on exiting  Weyl chambers. We  also find the tail asymptotics for the exit time of the random walk from  Weyl chambers. 
\end{abstract}

 \section{Introduction, assumptions and main results}
    \subsection{Introduction}
\label{intro}

The study of discrete harmonic functions and tail asymptotics for the exit time of a random walk in a cone involves a dependency between the moment conditions required for the increments of the random walk and the geometry of the cone~\cite{denisov_wachtel15,denisov_wachtel19,denisov_wachtel21}. This has not yet been fully understood, even in some of the most natural and simplest examples of cones. The present paper fills this gap for Weyl chambers of type C and D and, more broadly, harmonic polynomials related to reflection symmetries. 

Brownian motion in a Weyl chamber of type C is a collection of nonintersecting Brownian excursions while Brownian motion in a Weyl chamber of type D is a collection of nonintersecting Brownian motions with a reflecting boundary at zero. Both models have attracted significant interest in mathematical and statistical physics contexts. Brownian motions in Weyl chambers of type C and D were first introduced by Grabiner~\cite{grabiner}. They have been studied from various points of view such as modelling fibrous structures~\cite{dg}, wetting and melting~\cite{fisher}, connections to free fermions~\cite{schehr1}, Yang Mills theory~\cite{schehr2}, random matrix theory~\cite{tracy_widom, liechty_wang}, Brownian last-passage percolation~\cite{bfpsw}, positive temperature analogues \cite{nteka} and polymer endpoint distributions \cite{nguyen_remenik}. They are often called viscous walkers or watermelons with a wall. 

All of these works consider either Brownian motions or simple symmetric random walks with a boundary at zero. 
Since in this situation there is no overshoot one can obtain  an explicit 
expression for the positive harmonic function~\cite{konig2002}. 
In the  case of general random walks 
one has to consider some additional corrections.  
The discrete harmonic function and tail asymptotics for the exit time for general random walks in Weyl chambers of type C and D have been considered from two different directions. The first direction considered by~\cite{KonigWolfgang2010Rwct} was to extend the theory of ~\cite{denisov_wachtel10} from the Weyl chamber of type A to those of type C and D. This required an additional assumption that the increments of the random walk are symmetric in order that the Vandermonde determinant of type C or D respectively is harmonic for the free random walk. Without this symmetry assumption, this property is no longer true and presents one of the major technical obstacles. 
The second direction is to view random walks in Weyl chambers of type C and D as an example of the more general theory of random walks in cones developed in~\cite{denisov_wachtel15,denisov_wachtel19,denisov_wachtel21}. This approach allows the symmetry condition from~\cite{KonigWolfgang2010Rwct} to be removed but at the expense of needing to require non-optimal moment conditions. 
%In general cones, an invariance principle requires the random walks to have $p$ moments where $p$ is related to the minimal eigenvalue of the Laplace-Beltrami operator and where the cone determines the domain of the eigenvalue problem. This is optimal for general cones but not optimal when the cones have additional structure as is present in Weyl chambers of type C and D. 
%The aim of this paper is to fill this gap by studying random walks for a general class of increment distributions in Weyl chambers of type C and D under optimal conditions both in terms of moments and not imposing any symmetry requirements on the increment distribution.

For Brownian motion in cones, the tail asymptotics for the exit time from a cone were found in \cite{banuelos_smits, deblassie} where the exponent $p/2$  governing the decay is related to the minimal eigenvalue of a Laplace-Beltrami operator with the cone determining the domain of the eigenvalue problem. This was first extended to random walks in cones in~\cite{denisov_wachtel15} where the geometry of the cone was constrained by an extendability condition and the random walks were required to have $p$ finite moments. The geometry condition was later improved to include
convex cones~\cite{denisov_wachtel19} and $C^{1, \alpha}$ and starlike cones~\cite{denisov_wachtel21}. This theory has been extended in various other directions including the study of tail asymptotics for broader classes of processes such as products of random matrices \cite{grama} or related questions such as the study of the Green function \cite{duraj}. 

Although~\cite{denisov_wachtel15,denisov_wachtel19,denisov_wachtel21} require moment conditions on the increments of the random walk that are close to optimal for the generality of cones they consider, these moment conditions are not optimal when the cone has some additional structure. For example, for the Weyl chamber of type A, the works~\cite{denisov_wachtel15,denisov_wachtel19,denisov_wachtel21} require  $\E|X|^{\frac{d(d-1)}{2}}$ to be finite which is not the best condition; in~\cite{denisov_wachtel10} 
it is sufficient to have $\E|X|^{d-1}$ finite for $d\ge 3$. 
This relationship between moment conditions and geometry of the cone is the focus of this paper.
%Hence, although results about random walks in general cones do not require symmetry conditions of~\cite{KonigWolfgang2010Rwct} 
%they require existence  of additional moments for $C$ and $D$. 
With respect to the moment conditions required for the random walk, the distinguishing feature of Weyl chambers from general cones turns out to be that the positive harmonic function for Brownian motion killed at the boundary of the Weyl chamber is a polynomial (the Vandermonde determinant of corresponding type). Therefore we find that the natural context to study the relationship between geometry and moment conditions is within the class of cones where the positive harmonic function for Brownian motion killed at the boundary of the cone is a homogeneous harmonic polynomial. 

A further motivation for our studies is 
universality of limiting distributions 
for random walks in Weyl chambers of type $C$ and $D$. In a recent paper~\cite{DFW24} universality was shown for random walks in Weyl chamber of type $A$ when dimension $d$ grows sufficiently slowly. We hope that the present paper will provide a foundation for consideration of growing number of random walks in Weyl chambers of type $C$ and $D$. 

\subsection{Background and notation for ordered random walks}
  
Let $(X_{i}(j))_{ 1 \leq i \leq d}$ be independent  and identically distributed random variables 
distributed as $X$. 
The Weyl chamber of types $A$, $C$ and $D$ are defined as follows,
\begin{align*}
W_A &= \{(x_1, \ldots, x_d) \in \mathbb{R}^d : x_1 < x_2 <\ldots < x_d\}\\ 
W_C &=\{(x_1, \ldots, x_d) \in \mathbb{R}^d : 0<x_1 < x_2 < \ldots < x_d\}\\
W_D &= \{(x_1, \ldots, x_d) \in \mathbb{R}^d : |x_1| < x_2 <\ldots < x_d\}.
\end{align*}
Some of the arguments for these chambers coincide, so 
we will write $Z$ for $A$,  $C$ or $D$. Throughout we always take $d \geq 2$.

Consider a $d$-dimensional random walk $(x+S(n))_{n \geq 0} = (x_1+S_1(n), \ldots, x_d+S_d(n))_{n \geq 0}$ started from 
$x = (x_1,\ldots,x_d)$, where   
\[
S_i(n) = \sum_{j=1}^n X_{i}(j), \quad  n \geq 1, i \in \{ 1, \ldots, d\}. 
\]
For a starting point $x \in W_Z$ 
let 
\[
  \tau_x  := \inf\{n \geq 1\colon   x+S(n) \notin W_Z\}.
\]
be the exit time from $W_Z$.  

In this paper, we consider random walks conditioned to stay in the Weyl chambers $A,C$ or $D$. 
Traditionally, there are two approaches to studying conditioned random walks. 
The first approach is consider the behaviour of the random walks up to time $n$ 
given that the random walk stays in a Weyl chamber up to time $n$, that is given the event $\tau_x>n$.
For one-dimensional random walks, which corresponds either to $d=2$ and $Z=A$ or $d = 1$ and $Z = C$ in our setting,  
a corresponding functional central 
limit theorem showing convergence to the Brownian meander was proved in~\cite{Iglehart74} and~\cite{Bolthausen76}. 
The second approach is to consider the behaviour of the random walks
given that random walks stay in a Weyl chamber forever, that is 
given the event $\tau_x=\infty$ in some sense that needs to be defined. 
This can be achieved using a Doob $h$-transform, 
which requires finding a positive function $h=V^Z$ on $W_Z$ such that 
\begin{equation}\label{eq:rw.harmonicity}
\E_x[V^Z(x+ S(1)); \tau_x > 1]  = V^Z(x), \quad x \in W_Z.  
\end{equation} 
The limiting random walk can then be defined using the transformed measure 
$\hat{\P}$ and can be represented as,
\begin{multline*}
\hat{\P}(x+S(k)\in dy_k,k=1,\ldots,n) \\
=\P(x+S(k)\in dy_k,k=1,\ldots,n;  \tau_x>n)
\frac{V^Z(y_n)}{V^Z(x)}. 
\end{multline*}
For one-dimensional random walks,  
the corresponding functional central 
limit theorem showing convergence to the three-dimensional Bessel process was proved in~\cite{BJD06}. 
In these one-dimensional cases, one has 
the powerful Wiener-Hopf factorisation at their disposal to analyse the 
asymptotic behaviour of $\pr(\tau_x>n)$. In particular 
the harmonic function required by Doob's $h$-transform is given 
by a renewal process of ladder heights. 

For $d>2$ the case $Z=A$ was first considered in~\cite{denisov_wachtel10, eichelsbacher_konig}. 
These studies made use of  
the Vandermonde determinant 
\[
\Delta(x) = \prod_{1 \leq i < j \leq d} (x_j - x_i), 
\] which is a harmonic function for the free random walk.  
Then, under certain conditions, one can show that 
\[
 V^A(x) = \Delta(x) - \E\left[\Delta(x+S({\tau_x}))\right]
\]
is a positive harmonic function for the killed
random walk in the sense that it satisfies~\eqref{eq:rw.harmonicity}. 
Later it was shown in~\cite{KonigWolfgang2010Rwct} that under further symmetry 
assumptions on the random walk increments, for Weyl chambers of type $C$ and $D$ we can define similarly 
\begin{equation}\label{eq:vz.symmetry}
 V^Z(x) = h^Z(x) - \E\left[h^Z(x+S({\tau_x}))\right], \quad x\in W_Z
\end{equation}
and this satisfies~\eqref{eq:rw.harmonicity}.
Here $h^Z : W_Z \rightarrow \mathbb{R}_+$ is the positive harmonic function of Brownian motion killed at exiting $W_Z$ given correspondingly by 
\begin{equation*}%\label{eq:h.z}
    h^A(x)= \Delta(x), \quad h^D(x)= \prod_{1 \leq i < j \leq d} (x_j^2 - x_i^2)  \text{ and } \quad h^C(x)= h^D(x) \prod_{i=1}^d x_i.
\end{equation*}
The symmetry assumption in~\cite{KonigWolfgang2010Rwct}
ensures that $\E[h^C(x+S(n))]=h^C(x)$ and $\E[h^D(x+S(n))]=h^D(x)$ 
in a similar way to the case of the Weyl chamber of type $A$. One should note that without 
additional symmetry assumptions these equations are not true and the function~\eqref{eq:vz.symmetry} does not satisfy~\eqref{eq:rw.harmonicity}. Instead some further correction terms are required as can be seen in~\cite{denisov_wachtel15,denisov_wachtel19,denisov_wachtel21} for general cones.

%In the present paper we aim to address this deficiency and to study 
%asymptotic behaviour of random walks in Weyl chambers under optimal conditions. 

\subsection{Assumptions and main results}
For Weyl chambers, we will impose the following moment assumptions. 
\begin{assump}{(M)}(Moment assumptions)\label{ass-m}\ 
    \begin{enumerate}
        \item[(M1)]\label{ass-M1} $\E[X]=0$ and $\E[X^2]=1$, $i=1,\ldots, d$.
        \item[(M2)]\label{ass-M2}  %When $d=2$ and 
        In the case $d>3$ we additionally assume that $\E\left[|X|^{r_Z}\right] < \infty$ where $r_A=d-1$, $r_C=2d-1$ and $r_D=2d-2$. 
        When $d\le 3$ in the case of the Weyl chamber of type $A$ 
        and when $d=2$ in the case of the Weyl chamber of type $D$ we additionally  assume $\E\left[|X|^2\log(1+|X|)\right]<\infty$.

        %that is $\E\left[|X|^{d-1}\right]$  is finite if $x \in W_A$, and  $\E[|X|^{2d-1}]$  is finite if $x \in W_C$.
    \end{enumerate}   
\end{assump}    

\begin{remark}\label{rem:moments}
\begin{enumerate}[(i)]
    \item 
For the Weyl chamber of type $A$ for $d>3$ the moment assumptions are the same as 
in~\cite{denisov_wachtel10}. 
For $d=3$ we improve over~\cite{denisov_wachtel10},  where  existence of $\E[|X|^{2+\delta}]<\infty$ was required for some $\delta>0$. 
\item For Weyl chamber of types $C$ and $D$ the moment assumption are the same as in~\cite{KonigWolfgang2010Rwct} apart from the case $d=2$ for the Weyl chamber of type $D$, where 
existence of $\E[|X|^{2+\delta}]<\infty$ was required for some $\delta>0$. 
The main improvement over~\cite{KonigWolfgang2010Rwct} is that we do not need 
the symmetry assumption $\E[X^r]=0$ for all odd $r\le r_Z$ used in~\cite{KonigWolfgang2010Rwct}. 
%The symmetry assumption in~\cite{KonigWolfgang2010Rwct}
%ensures that $\E[h^C(x+S(n)]=h^C(x)$ and $\E[h^D(x+S(n)]=h^D(x)$ 
%in a similar way to the case of the Weyl chamber of type $A$. This is not the case without a symmetry assumption.
%Then it was shown in~\cite{KonigWolfgang2010Rwct} that 
% the methodology of~\cite{denisov_wachtel10} is applicable 
% and, in particular, equation~\eqref{eq:vz.symmetry} is a solution to~\eqref{eq:rw.harmonicity}.  
%In the present paper we use an approach similar to~\cite{denisov_wachtel21} 
%that allows us to get rid of the symmetry assumptions while retaining optimal moment conditions. 
\item
In general cones \cite{denisov_wachtel15, denisov_wachtel19, denisov_wachtel21}, the random walks are required to have $p$ moments where $p$ is related to the minimal eigenvalue of the Laplace-Beltrami operator and where the cone determines the domain of the eigenvalue problem. The value of $p$ in each of the Weyl chambers is given by $d(d-1)/2$ for type A, $d(d-1)$ for type D and $d^2$ for type C. The assumptions on the cone in \cite{denisov_wachtel15, denisov_wachtel19} are satisfied by Weyl chambers of type C and D as they have the required extendability condition or are convex respectively. The condition on the cone in \cite{denisov_wachtel21} is not satisfied, as this requires $C^{1, \alpha}$ smoothness of the cone.   
\item
The assumption $\E[|X|^2\log(1+|X|)]<\infty$ might be 
in some sense optimal, see discussion and examples in \cite[Section 1.3]{denisov_wachtel21} for the Weyl chamber of type $D$ and $d=2$. 
\end{enumerate}
 \end{remark}

 The results for Weyl chamber will follows from a more general setting of 
 homogeneous harmonic polynomials. 
Let $h : \mathbb{R}^d \rightarrow \mathbb{R}$ be a homogeneous harmonic polynomial, that is a polynomial satisfying 
\[
\sum_{j=1}^d\frac{\partial^2 h}{\partial x_j^2}=0
\]
and $h(tx) = t^p h(x)$ where $p$ is the degree of the polynomial. 
We also assume that $h$ has the following form for a sequence of vectors $\alpha_1, \ldots, \alpha_p \in \mathbb{R}^d$
\[
h(x) =  \prod_{i=1}^p \langle x, \alpha_i \rangle.
\]
We consider the cone $K$ of the following form 
\begin{equation}\label{cone.polynomial}
K=\{x\colon \langle x,\alpha_i\rangle>0\}. 
\end{equation}
It follows immediately that $K$ is convex. 
Clearly these assumptions are met by Weyl chambers of types $A,C$ and $D$.

% These cones are convex and star-like with $x_0=(0,\ldots, d-1)$ and the harmonic polynomial has the required form. 

% The set $\{x \colon h(x)>0\}$  is a union of connected cones since the polynomial is homogeneous. 
% Let $K$ be one of these connected cones and define 
% \[
% \tau_x:=\inf\{n\ge 1\colon x+S(n)\in K\}. 
% \]

Since $K$ is convex 
    there exists 
    $x_0\in \Sigma=K\cap\{x\colon |x|=1\}$ such that $x_0 + K \subset K$ and ${\rm dist}(x_0+K, \partial K)>0$. 
    In the rest of the text $x_0$ is used to denote a vector satisfying this property.   
    
Let $p$ be the degree of the polynomials $h$ and 
let $r$ be the maximum of degree of all variables in the polynomial. 
For example, for $h(x)=(x_3-x_2)(x_3-x_1)(x_2-x_1)$ we have $p=3$ and $r=2$. 

Let $\delta(x):=\text{dist}(x,\partial K).$
We will summarise our assumption as follows 
on cone $K$ and moments of $X$ below. 
\begin{assump}{(P)}(Assumptions in the polynomial case)\label{ass-m-h}\ 
    \begin{enumerate}
        \item[(P1)]\label{ass-M1-h} Suppose that $\{X_i(j),i=1,\ldots,d,j=1,2,...\}$ is a sequence of $i.i.d.$ random variables
        with common distribution $X$. Suppose that $\E[X]=0$ and $\E[X^2]=1$. 
        \item[(P2)]\label{ass-M2-h}  
        For $r=2$ we assume $\E\left[|X|^2\log(1+|X|)\right]$. 
        For $r>2$ we assume $\E\left[|X|^{r}\right] < \infty$. 
        \item[(P3)] Suppose the harmonic function $h$ is a positive homogeneous polynomial of degree $p$ vanishing at the boundary of $K$. Furthermore, we assume  that 
        there exists a constant $C>1$ such that 
        \begin{equation}\label{eq:harmonic.boundary}
           h(x) \le  C|x|^{p-1}\delta(x),\quad x\in K. 
        \end{equation}
        \item[(P4)] Suppose that the harmonic polynomial has the form for a sequence of vectors $\alpha_1, \ldots, \alpha_p \in \mathbb{R}^d$
\[
h(x) =  \prod_{i=1}^p \langle x, \alpha_i \rangle.
\]
        \end{enumerate}   
\end{assump}

%\begin{theorem}
%Let Assumption \ref{ass-m} hold and let $Z \in \{A, C, D\}$ denote the type of the Weyl chamber.
% Then the function 
%\begin{equation}\label{eq:harmonic}
%    V^Z(x) := \lim_{n \to \infty} \E \left[h^Z(x+S(n); \tau_x>n\right] 
%\end{equation}
%is finite, positive and harmonic 
%for $(x+S(n))_{n \geq 0}$ killed on leaving the Weyl chamber of type $Z$, that is 
%\begin{equation}
%\label{eq:harmonic_property}
%    V^Z(x) =\E [V^Z(x+S(1)); \tau_x>1],\quad    
%    x\in W_Z. 
%\end{equation}
%Furthermore,
%    \begin{equation}\label{eq:renewal}
%    V^Z(x)\sim h^Z(x), \quad\text{for }x\in W_Z, \text{ as } \delta_Z(x)\to\infty,
%    \end{equation}
%    and uniformly in $x \in W_Z$ such that $\delta_Z(x)=o(\sqrt n)$ for $Z=A$,\textcolor{blue}{(other Weyl chambers? is this proved?)}
%    \[
        %\sup_{x\in K: |x|=o(n^{1/2})}
%        \left| 
%        \frac{\E \left[h^Z(x+S(n));\tau>n\right] - V^Z(x)}{1+h^Z(x)}
%        \right| 
%        \to 0,\quad  n\to \infty. 
%    \]
%\label{thm:harmonic.function}
%\end{theorem}

\begin{theorem}
\label{thm:harmonic.function-poly}
Let Assumption~\ref{ass-m-h} hold. 
 Then the function 
\begin{equation}\label{eq:harmonic}
    V(x) := \lim_{n \to \infty} \E \left[h(x+S(n); \tau_x>n\right] 
\end{equation}
is finite, positive and harmonic 
for $(x+S(n))_{n \geq 0}$ killed on leaving $K$, that is 
\begin{equation}
    V(x) =\E [V(x+S(1)); \tau_x>1],\quad    
    x\in K. 
\end{equation}
% Furthermore,
%     \begin{equation}\label{eq:renewal}
%     V(x)\sim h(x), \quad\text{for }x\in K, \text{ as } \delta(x)\to\infty,
%     \end{equation}
%     and uniformly in $x \in K$ such that $\delta(x)=o(\sqrt n)$,
%     \textcolor{red}{is this proved? Not proved yet. I think there is no harm to  simply exclude this statement}
%     \[
%         %\sup_{x\in K: |x|=o(n^{1/2})}
%         \left| 
%         \frac{\E \left[h(x+S(n));\tau>n\right] - V(x)}{1+h(x)}
%         \right| 
%         \to 0,\quad  n\to \infty. 
%     \]
% \label{thm:harmonic.function-poly}
\end{theorem}

\begin{remark}
One can show that $V(x)\sim h(x)$ as $\delta(x)\to \infty$. We will omit this property since similar computations have already been done in e.g.~\cite{denisov_wachtel21}. 
\end{remark}

We now consider tail asymptotics for the exit time from the cone. For Weyl chambers of type A these questions were studied in \cite{denisov_wachtel10, eichelsbacher_konig}, for type C and D in~\cite{KonigWolfgang2010Rwct} and for general cones in~\cite{denisov_wachtel15,denisov_wachtel19,denisov_wachtel21}. 
In comparison with these papers we have improved moment and symmetry assumptions, see Remark~\ref{rem:moments}. 

%\begin{theorem}
%\label{thm.p.tau}
%Let Assumption \ref{ass-m} hold.
%Let $Z \in \{A, C, D\}$ denote the Weyl chamber with $p=d(d-1)/2$ if $Z=A$,  %$p=d^2$ if $Z=C$ and $p=d(d-1)$ if $Z=D$. 
%\begin{enumerate}[(i)]
%\item  There exists  constants $C$ and $R$ such that,
%\begin{equation}
%    \P (\tau_x>n) \le C \frac{h^Z(x+Rx_0)}{n^{\frac{p}{2}}}, 
%    \quad x\in W_Z. 
%\end{equation}
%\item Uniformly in $x$ such that $\max\limits_{2\le j \le d}(x_j-x_{j-1})=o(\sqrt %n)$ for $Z=A$, or $\max\limits_{2\le j \le d}(x_{j-1})=o(\sqrt n)$ for $Z=C$ %respectively,
%\textcolor{blue}{(type D?)}
%\begin{equation}
%    \P (\tau_x >n) \sim \varkappa \frac{V^Z(x)}{n^{\frac{p}{2}}}, 
%    \quad n\to \infty. 
%\end{equation}
%\item 
%For any compact set $K \subset W_Z$, with $\max\limits_{2\le j \le d}(x_j-x_{j-1}, %x_{j-1})=o(\sqrt n)$ for $Z=A$, or $\max\limits_{2\le j \le d}(x_{j-1})=o(\sqrt n)$ %for $Z=C$ respectively,
%\begin{equation}
%    \P \left( \frac{x+S(n)}{\sqrt{n}} \in K \mid \tau_x >n \right) \to  c\int_D %h^Z(z)\exp\left\{-\frac{|z|^2}{2}\right\} d z, 
%    \quad n\to \infty. 
%\end{equation}
%\end{enumerate}
%\end{theorem}

\begin{theorem}\label{thm.p.tau-poly}
Let Assumption~\ref{ass-m-h} hold. 
\begin{enumerate}[(i)]
\item  There exist  constants $C$ and $R$ such that,
\begin{equation}
    \P (\tau_x>n) \le C \frac{h(x+Rx_0)}{n^{\frac{p}{2}}}, 
    \quad x\in K. 
\end{equation}
\item Uniformly in $x$ such that $|x|=o(\sqrt n)$, there is a constant $\varkappa$ such that
\begin{equation}
    \P (\tau_x >n) \sim \varkappa \frac{V(x)}{n^{\frac{p}{2}}}, 
    \quad n\to \infty. 
\end{equation}
\item 
For any compact set $D \subset K$, with $|x|=o(\sqrt n)$, there is a constant $c$ such that
\begin{equation}
    \P \left( \frac{x+S(n)}{\sqrt{n}} \in D \mid \tau_x >n \right) \to  c\int_D h(z)\exp\left\{-\frac{|z|^2}{2}\right\} d z, 
    \quad n\to \infty. 
\end{equation}
\end{enumerate}
\end{theorem}

The constants $\varkappa, c$ depend on $K$ but do not depend on the increment distribution in the random walk.

\begin{remark}
To illustrate why $r$ moments are required, we can consider two homogeneous harmonic polynomials of degree 4, 
\begin{align*}
    \phi_1(x, y) = xy(x^2 - y^2), \qquad \phi_2(x, y) = x^4 - 6x^2y^2 + y^4.
\end{align*}
Let $K_1$ and $K_2$ be arbitrarily chosen connected components of $\phi_1(x, y) > 0$ and 
$\phi_2(x, y) > 0$ respectively.
Theorem~\ref{thm.p.tau-poly} establishes in both settings that 
$\P(\tau_x > n) \sim \varkappa V(x)n^{-2}$ and the cones are related by a rotation. However, for $K_1$ the moment condition required on the random walk increments is $\E[X^3] < \infty$ while for $K_2$ the moment condition is
$\E[X^4] < \infty.$ These conditions are optimal; the reason for this difference can be seen from the 
one-big-jump-principle, in the case of $K_2$ one big jump can take the random walk into the interior of the cone, whereas, with $K_1$ two big jumps are required. This comes from the fact that $K_1$ has boundaries along the co-ordinate axes while $K_2$ does not.
\end{remark}

%This approach covers also star-like Lipschitz cones $K$ 
%satisfying the assumption below. 
%In this case the coordinates of the vector $X$ might be dependent. 
%We list the required assumptions below. \textcolor{blue}{(I think this Theorem needs more explanation about why it is true. Equation (5) of \cite{denisov_wachtel21} is used in a few places in the proof.)}
%\begin{assump}{(L)}(Assumptions for general  case)\label{ass-m-c}\ 
%    \begin{enumerate}
%        \item[(L1)]\label{ass-M1-c} $\E[X_i]=0$ and $\E[X_i^2]=1$, $i=1\ldots d$.
%        $(X(j))_{j\ge 1}$ is a sequence of i.i.d. uncorrelated random vectors, that is $\E[X_iX_j]=0$ when $i\neq j$. 
%        \item[(L2)]\label{ass-M2-c}  
%        For $r=2$ we assume $\E\left[|X|^2\log(1+|X|)\right]$. 
%        For $r>2$ we assume $\E\left[|X|^{p}\right] < \infty$. 
%        \item[(L3)]\label{ass-M3-c}  Harmonic function $h$ is positive homogeneous of  degree $p$ vanishing at the boundary of $K$. 
%        Furthermore, we assume  that 
%        there exists a constant $C>1$ such that~\eqref{eq:harmonic.boundary} holds.   
%        \end{enumerate}   
%\end{assump}    

The present paper relies on the methodology 
developed in~\cite{denisov_wachtel21}, in particular that we can combine estimates on the Green function of Brownian motion with estimates of the error in the diffusive approximation. 
There are two main technical obstacles in extending this methodology. Firstly, \cite{denisov_wachtel21} requires  
$C^{1, \alpha}$ smoothness of the cones and an extra condition that  
\begin{equation}\label{eq:harmonic.boundary.lower}
           h(x) \ge  C|x|^{p-1}\delta(x),\quad x\in K. 
\end{equation}  
This is used in a number of places, for example to establish estimates on the Green's function. Secondly, we assume weaker moment conditions than in \cite{denisov_wachtel21} which means we need to establish more precise estimates by using the specific structure of the harmonic polynomial.
%Secondly, there is a complication arising for the Weyl chamber of type D when the first particle takes negative values. The motion of the first particle can be thought of as being reflected ($x_2 > \lvert x_1 \rvert$) but due to the lack of a symmetry condition the motion of the first particle is then not the same as the motion of the other particles. 
%This property adds in new complications and makes the asymptotic equivalence %between the harmonic function for the random walk and the harmonic function for Brownian motion more involved.

The rest of the paper is structured as follows. In Section \ref{sec:preliminary_estimates} we establish a number of preliminary estimates on the Green's function of killed Brownian motion and the error arising from taking the expectation of the harmonic function for Brownian motion after one step of the killed random walk. In Section \ref{sec:superharmonic} we construct a superharmonic function.
In Section~\ref{sec:harmonic} we construct the harmonic function and prove Theorem~\ref{thm:harmonic.function-poly}. In Section \ref{sec:tail_asy_bound} we prove upper bounds for the tail asymptotics. In Section~\ref{sec:tail_asymptotics} we prove Theorem~\ref{thm.p.tau-poly}. 

%\begin{remark}
%Note that~\eqref{eq:harmonic.boundary} holds 
%when $K$ is convex or $C^{1,\alpha}$. 
%In comparison with~\cite{denisov_wachtel21} and~\cite{denisov2023markov} 
%we have included the important convex case. \textcolor{blue}{(compare to %\cite{denisov_wachtel19}?)} 
%In those papers an extra condition that  
%\begin{equation}\label{eq:harmonic.boundary.lower}
%           u(x) \ge  C|x|^{p-1}\delta(x),\quad x\in K. 
%\end{equation}        
%for some positive $C$ was imposed. While this condition 
%is satisfied by~$C^{1,\alpha}$ cones it does not hold 
%in general for convex cones. 
%This condition was included to ensure that an estimate. 
%for the Green function of Brownian motion hold. 
%In the present paper we have shown that the same estimate 
%holds for general Lipschitz cones, see Lemma~\ref{lem:green.function.bound0.poly} below.  
%This is the only place, where~\eqref{eq:harmonic.boundary.lower} was used \textcolor{blue}{(It looks like it is other places too such as Lemma 15, proving Equation 56 in the construction of the harmonic function, Lemma 22 and Lemma 3 on upper bounds for tail asymptotics and just above equation 56 in asymptotic equivalence of harmonic functions)}. 
%Hence, the proofs presented in~\cite{denisov_wachtel21} and~\cite{denisov2023markov} will work under the Assumption~\ref{ass-m-c} and there is no need to repeat them. Instead we will focus on the polynomial case. 
%\end{remark}

\section{Preliminary estimates} 
\label{sec:preliminary_estimates}
In this section, we will derive the bound of error in the diffusion approximation of harmonic function of $(x+S(n))_{n \geq 0}$ with harmonic function $h$ of Brownian motion. We then prove an upper bound on the Green's function of Brownian motion.

We start with the following estimate for harmonic functions. 

\begin{lemma}
\label{lem:harnack}
There exists a constant $C$ such that 
\begin{equation}\label{eq:harnack}
\left| \frac{\partial^k h(x)}{\partial x_1^{\alpha_1}\cdots\partial x_d^{\alpha_d}}\right| \le C\frac{h(x)}{\delta(x)^k}, \quad x\in K,
\end{equation}
where $k=\alpha_1+\cdots+\alpha_d$.
\end{lemma}
\begin{proof}
    This lemma follows from the Harnack inequality. 
    It is is essentially Lemma~2.1 from~\cite{denisov_wachtel19}, where it was proved for $k\le 3$. 
    It is immediate that 
    the same proof works for $k\le p$. Alternatively, one can use the explicit  representation for $h(x)$ and perform direct computations. Note that for $k>p$ the partial derivatives vanish, so the statement is trivially correct in this case.  
\end{proof}

We will also need the following. 
\begin{lemma}\label{lem:construction.gamma}
Let the Assumption~\ref{ass-M2-h} hold. Then, there exists a slowly varying, monotone decreasing differentiable function $\gamma(t)$ such that  
    \begin{equation}\label{eq:gamma.integral.finite}
        \E[|X|^{r};|X|>t]=o(\gamma(t)t^{r-2}),\quad t\to\infty, 
    \end{equation} 
    and 
    \begin{equation}    
        \label{eq:gamma.integral.finite.3}
        \int_1^\infty x^{-1}\gamma(x) dx<\infty.
    \end{equation}
\end{lemma}    
\begin{proof}
The proof can be found in~\cite[Lemma 10]{denisov_wachtel21}. 
The main subtlety is the case $r=2$, when the proof follows  from the existence of an integrable regularly varying majorant for a  monotone integrable function~\cite{denisov06}. 
\end{proof}

\subsection{Estimate for the error term and the Green function of Brownian motion}
For $x\in K$  let  
\begin{equation}\label{eq:defn.f}
f(x) := \E\left[ h(x+X), x+X\in K\right] -h(x). 
\end{equation}
\begin{lemma}\label{lem:bound.f}
  Let $f$ be defined by~\eqref{eq:defn.f} and let Assumption \ref{ass-m-h} hold.  Then, 
\begin{equation*} 
    |f(x)|=o( \beta(x)),\quad \text{ as }\delta(x)\to\infty,   
  \end{equation*}
where 
\begin{equation}\label{defn.beta}
  \beta(x):= h(x)\int^{\infty}_{\delta(x)} \frac{\gamma(t)}{t^3} dt
\end{equation}
and function $\gamma(x)$ is a slowly varying, monotone decreasing differentiable function as 
in~Lemma~\ref{lem:construction.gamma}. 
\end{lemma}
\begin{proof}
Applying the triangle inequality to \eqref{eq:defn.f} we obtain 
\begin{equation}\label{eq:f}
   |f(x)|\le g_1(x)+g_2(x), 
\end{equation}
where 
\begin{align*} 
    g_1(x)&:=\left|\E\left[h(x+X)-h(x)\right]\right|  \\
    g_2(x)&:=\left|\E\left[h(x+X), x+X \notin K\right]\right|. 
\end{align*}
For the event $\{x+X \notin K\}$ to happen the random variable should satisfy $|X|\ge\delta(x)$.  
This in turn implies that $|X_i|\ge \frac{\delta(x)}{\sqrt {d}}$ at least for one of $i\in \{1,\ldots,d\}$ and  
the following inclusion is valid 
\[
\left\{x+X \notin K\right\} \subseteq \bigcup_{i=1}^d \left\{|X_i| >\frac{\delta(x)}{\sqrt{d}}\right\}. 
\]
Let $J_i:=\left\{|X_i| >\frac{\delta(x)}{\sqrt{d}}\right\}$ and note that  
\begin{align*}
    |g_2(x)| &\le \Bigg\lvert \E\left[h(x+X);  \bigcup_{i=1}^d \left\{|X_i| \ge \frac{\delta(x)}{\sqrt{d}}\right\} \right] \Bigg\rvert
    \\& \le \sum_{i=1}^{d} \big\lvert \E\left[h(x+X); J_i\right] \big\rvert. 
\end{align*}
Next, by the Taylor formula, 
\[ 
  h(x+X)=\sum_{k=0}^{p} \frac{1}{k!} \sum_{\alpha_1+\cdots+\alpha_d = k}\frac{\partial^k h(x)}{\partial x_1^{\alpha_1} \cdots \partial x_d^{\alpha_d}} X_1^{\alpha_1} \cdots X_d^{\alpha_d}. 
\] 
Note that when $k>p$, all  partial derivatives of $h(x)$ are equal to 0, so this formula is exact. 
Also note that the partial derivatives vanish if 
$\alpha_j>r$ for some $\alpha_j$, hence all the coefficients appearing below are finite by Assumption \ref{ass-m-h} .  
Then, for $i=1,\ldots, d$, 
\[
   \E[h(x+X); J_i]= 
   \sum_{k=0}^{p} \frac{1}{k!} \sum_{\alpha_1+\cdots+\alpha_d =  k}\frac{\partial^k h(x)}{\partial x_1^{\alpha_1} \cdots \partial x_d^{\alpha_d}} \prod_{j\neq i}\E[X_j^{\alpha_j}]  \E[X_i^{\alpha_i}; J_i]. 
\]
Next, for $\alpha_i\le r$ we have,  
\[
|\E[X_i^{\alpha_i}; J_i]| \le 
\frac{\E[|X_i|^{r}; J_i]}{\delta(x)^{r-\alpha_i}}. 
\]
Then, by~\eqref{eq:gamma.integral.finite} and Lemma~\ref{lem:harnack}, 
as $\delta(x)\to \infty$, 
\begin{align*}
\left|\E[h(x+X); J_i]\right|
    &\le 
    \sum_{k=0}^{p} \frac{C}{k!} \sum_{\alpha_1+\cdots+\alpha_d =  k, \alpha_1, \ldots, \alpha_d \leq r_z}
    \frac{ h(x)}{\delta(x)^k}
    %\frac{\partial^k h^Z(x)}{\partial x_1^{\alpha_1} \cdots \partial x_d^{\alpha_d}} 
    \prod_{j\neq i}\E[|X_j|^{\alpha_j}]  
    \frac{\E[|X_i|^{r}; J_i]}{\delta(x)^{r-\alpha_i}}
    \\
    &\le  
    h(x) o\left(\frac{\gamma(\delta(x))}{\delta(x)^{2}}\right)
  +
    C \sum_{k=3}^{p} \sum_{\alpha_1+\cdots+\alpha_d =  k}\frac{ h(x)}{\delta(x)^{k+r-\alpha_i}}\\
   &\le 
   h(x) o\left(\frac{\gamma(\delta(x))}{\delta(x)^{2}}\right)
   +Ch^Z(x) \frac{1}{\delta(x)^3} \\
   & =o(\beta(x)),
\end{align*}
where we also used the fact that by the Karamata theorem, $\frac{\gamma(t)}{t^2}\sim \frac{1}{3}\int_t^\infty \frac{\gamma(u)}{u^3}du.$ 

Hence, 
\begin{equation}\label{eq.f.1}
    g_2(x)=o(\beta(x)). 
\end{equation}

Next, we estimate the term $g_1(x)$. By the Taylor formula,
\begin{multline*}
  \E\left[h(x+X)-h(x)\right] = \sum_{j=1}^d \frac{\partial h(x)}{\partial x_i} \E\left[X_j\right]+\frac{1}{2}\sum_{1\le i,j \le d}\frac{\partial^2 h(x)}{\partial x_i \partial x_j} \E\left[X_i X_j\right] \\+ \sum_{k=3}^{p} \frac{1}{k!} \sum_{\alpha_1+\cdots+\alpha_d = k}\frac{\partial^k h(x)}{\partial x_1^{\alpha_1} \cdots \partial x_d^{\alpha_d}} \E\left[X_1^{\alpha_1} \right] \cdots \E\left[X_d^{\alpha_d} \right]. 
\end{multline*}
Recall that $\E[X_j]=0$, $\mathrm{cov}(X_i,X_j)= \delta_{i=j}$ and $h(x)$ is harmonic, 
that is  $\sum_{j=1}^d \frac{\partial^2 h(x)}{\partial x_j^2} =0 $. Then, using Lemma~\ref{lem:harnack}, 
\begin{multline*}
  \left|\E\left[h(x+X)-h(x)\right]\right| \\
  \le C \sum_{k=3}^{p} \frac{1}{k!} \sum_{\alpha_1+\cdots+\alpha_d = k, \alpha_1, \ldots, \alpha_d \leq r_z} \frac{h(x)}{\delta(x)^k}
  \E\left[|X_1|^{\alpha_1} \right] \cdots \E\left[|X_d|^{\alpha_d} \right]
  \le C\frac{h(x)}{\delta(x)^3}.  
\end{multline*}
Since  a slowly varying function increases slower than 
any power function we obtain 
\begin{equation}\label{eq.f.2}
g_1(x) = o(\beta(x)).  
\end{equation}
Combining equations~\eqref{eq.f.1} and \eqref{eq.f.2} we arrive at the conclusion. 
\end{proof}
Next, we are planning to construct 
a suitable positive superharmonic function 
similarly  to~\cite{denisov_wachtel21}. 
For that we will make use of the following estimate of the the Green function $G(x,y)$ of the Green function of Brownian motion killed on exiting $K$. 
\begin{lemma}\label{lem:green.bound0} 
Let $K$ be of the form~\eqref{cone.polynomial}. Then 
there exists a constant $C$ such that 
  \begin{equation}
  \label{eq:green.function.bound0}
  G(x,y)  \le C \widehat G(x,y), 
  \end{equation} 
  % \begin{cases}
  %     \frac{u(x)u(y)}{|y|^{d-2+2p}},& |x|\le \frac{4}{5}|y|\\
  %     \frac{u(x)u(y)}{|x|^{d-2+2p}},& |x|\ge \frac{5}{4}|y|\\ 
  %     \frac{u(x)u(y)}{|x|^{d-2+p}|y|^p} \left(1+\frac{|y||x|^{d-1}}{|x-y|^d}\right)
  %                                   &\frac{4}{5}|y|\le |x| \le \frac{5}{4}|y|
  % \end{cases}
  where 
  \begin{equation}\label{eq:ghat}
  \widehat G(x,y)  := %\frac{h^Z(x)h^Z(y)}{\max(|x|,|y|)^{2p-2}|x-y|^d}.
  \begin{cases}
  \frac{h(x)h(y)}{|y|^{2p+d -2}},& |x|\le |y|, |x-y|\ge \frac{1}{2}|y|,\\
   \frac{h(x)h(y)}{|x|^{2p+d -2}},& |y|\le |x|, |x-y|\ge \frac{1}{2}|y|,\\ 
   \frac{h(y)(\delta(x))^2}{h(x)|x-y|^{d}},
                                 &\delta(y)/2\le |x-y| \le \frac{1}{2}|y|,\\
                                 \frac{1}{|x-y|^{d-2}}I(d>2)
                                 +\ln\left(\frac{\delta(y)}{|x-y|}\right)I(d=2), 
                                 & |x-y|<\delta(y)/2.
  \end{cases}
  \end{equation}
%and $p=d(d-1)/2$ when $Z=A$,  $p=d^2$  when $Z=C$ and 
%$p=d(d-1)$ when $Z=D$. 
\end{lemma}
  \begin{proof}    We will make use of \cite[Theorem 3.2]{hirata06}, which 
    states that 
    \begin{equation}
    \label{hirata}
      G(x,y) \le C \frac{h (x)h(y)}{(h(b))^2 |x-y|^{d-2}} 
    \end{equation}
    for $x,y \in K$ and $b\in \mathcal B(x,y)$. 
    Here, for some $\kappa\ge 1$, 
    \[
    \mathcal B(x,y) 
    := \left\{
        b \in K\colon 
        \max(|x-b|,|y-b|)\le \kappa|x-y| \text{ and }
        \delta(b) \ge \frac{|x-y|}{\kappa} 
      \right\}.  
    \]
   Recall that $h(x) \geq C\delta(x)^p$; this is proved in Lemma 2.2 of \cite{denisov_wachtel19} when $K$ is convex. 
    Recall that from Assumption (P4) the harmonic polynomial has the form for a sequence of unit vectors $\alpha_1, \ldots, \alpha_p$
\[
h(x) =  \prod_{i=1}^p \langle x, \alpha_i \rangle.
\]
Then without loss of generality we can consider 
\[
K = \{x \in \mathbb{R}^d : \langle x, \alpha_i \rangle\} > 0
\]
which is a convex cone.
Choose $x^0$ such that $\langle x^0, \alpha_i \rangle \geq 1$.    We choose \[
b = x + \lvert y - x\rvert x^0
\] which satisfies the requirements that $b \in \mathcal{B}(x, y)$.
    
  We start with the  first line in~\eqref{eq:ghat}. 
    Then
    $h(b)\ge C\delta(b)^p\ge C' |x-y|^p$ by the choice of $b$. 
    In turn, since $|x-y|\ge \frac{1}{2}|y|$, we obtain the required result. 
    
    To prove the second line in~\eqref{eq:ghat}, we consider  two cases. 
    First, when  $|x|/2\le |y|\le |x|$ we use the same argument to obtain 
    that 
    \[
      h(b)\ge C\delta(b)^p
      \ge C' |x-y|^p 
      %\ge \kappa^{-p} 2^{-p} |y|^p
      \ge C'' |x|^p 
    \]
    since $|x-y|\ge \frac 12 |y|\ge \frac{1}{4}|x|$. 
    When $|y|\le |x|/2$ note first that  
    $|x-y|\ge |x|-|y|\ge |x|/2$. Then, 
    \[
      h(b)\ge C\delta(b)^p\ge C' |x-y|^p 
      \ge C'' |x|^p.
    \]
    This proves the second line.

To prove the third line in~\eqref{eq:ghat} choose $k$ such that $\langle x, \alpha_k\rangle = \min_{i=1}^p \langle x, \alpha_i\rangle$.
Note that $\langle \lvert y - x \rvert x^0, \alpha_i \rangle > 0$ for all $i \neq k$.
Therefore,
\begin{align*}
\frac{h(x)}{h(b)} & = \prod_{i=1}^p \frac{\langle x, \alpha_i \rangle}{\langle x + \lvert y - x \rvert x^0, \alpha_i \rangle} \\
& \leq \frac{\langle x, \alpha_k \rangle}{\langle x + \lvert y - x \rvert x^0, \alpha_k \rangle} \\
& \leq \frac{\langle x, \alpha_k \rangle}{\lvert y - x \rvert \langle x^0, \alpha_k \rangle} \\
& \leq \frac{\langle x, \alpha_k \rangle}{\lvert y - x \rvert} \\
&  = \frac{\delta(x)}{\lvert y - x \rvert}. 
\end{align*}
Using this in \eqref{hirata} gives the third line.
    The fourth line follows by estimating the Green function 
    for Brownian motion in a cone by the Green function in the whole space.  
  \end{proof} 
%\section{Construction of harmonic function} 
\section{Construction of superharmonic function}
\label{sec:superharmonic}
For $x\in K$,  
let  $\beta(x)$ be the function defined in~\eqref{defn.beta}. 
For $y\in K$ let 
\begin{equation}\label{eq:ubeta}
U_\beta(y) := \int_{K} G(x,y) \beta(x) dx.  
\end{equation}
By the definition of the Green function 
\begin{equation}\label{eq:laplacian.drift}
\Delta U_\beta(y) = -\beta(y).
\end{equation}

Using Lemma~\ref{lem:green.bound0} and the definition of $\beta$ in Lemma \ref{lem:bound.f} one can show that $U_\beta$ is well defined.  
For that we will follow the approach in~\cite{denisov_wachtel21} and 
estimate the integral in~\eqref{eq:ubeta} using a sequence of Lemmas.
To make the comparison to~\cite{denisov_wachtel21} it is useful to observe that
\[
h(x) \leq C\delta(x) \lvert x \rvert^{p-1} 
\]
so that 
\[
\beta(x) = o\left(\frac{\lvert x \rvert^{p-1} \gamma(\delta(x))}{\delta(x)}\right).
\]

\begin{lemma}\label{lem:ghat.1}
There exists a function $\varepsilon_1(R)\to 0$ such that for $y\in  K: |y|>R$,
\begin{equation}\label{eq:ghat.1}
I_1(y):=\int_{K \cap \{|x|\le |y|, |x-y|\ge \frac 12 |y|\}}\widehat G(x,y)\beta(x) dx 
\le \varepsilon_1(R) h(y).
\end{equation}
\end{lemma}    
\begin{proof}
  The proof is the same as  the proof of Lemma~12 in 
  \cite{denisov_wachtel21}. 
\end{proof}

\begin{lemma}\label{lem:ghat.2}
There exists a function $\varepsilon_2(R)\to 0$ such that for $y\in K: |y|>R$,
\begin{equation}\label{eq:ghat.2}
I_2(y):=\int_{K \cap \{|x|\ge |y|, |x-y|\ge A |y|\}}
\widehat G(x,y)\beta(x) dx \le \varepsilon_2(R) h(y).
\end{equation}
\end{lemma}    
\begin{proof}
  The proof is the same as  the proof of Lemma~13 in 
  \cite{denisov_wachtel21}. 
  %This is due to the fact that the first and the second 
  %estimate in~\eqref{eq:ghat} are the same as in Lemma 9 in~\cite{denisov_wachtel21}. 
   \end{proof}

\begin{lemma}\label{lem:ghat.3}
There exists a bounded monotone 
function $\varepsilon_3(R)\to 0$ such that for $y\in K: d(y)>R$,
\begin{equation}\label{eq:ghat.3}
I_3(y):=
\int_{K \cap \{\delta(y)/2<|x-y|<\frac{1}{2}|y|\}}\widehat G(x,y)\beta(x) dx \le \varepsilon_3(R) h(y).
\end{equation}
\end{lemma}    
\begin{proof}
%First note that $|x-y|<\frac{1}{2}|y|$ with $A<1$ implies that 
% \[
% \frac 12 |y|<|x|<\frac 32|y|. 
% \]    
% We have, 
% \[
% I_3(y)\le 
% \sum_{n=0}^{[\log_2(|A|y|/d(y))]}    
% \int_{K \cap \left\{2^{n-1}d(y)<|x-y|<
% 2^{n}d(y)\right\}}
% G(x,y)\beta(y) dy 
% \]
% Consider $x\in K \cap \left\{2^{n-1}d(y)<|x-y|<
% 2^{n}d(y)\right\}$. 
% When $d(x)<2^{n-3}d(y)$ there exists 
% $x^*$ at the boundary such that 
% $|x-x^*|=d(x)$.
% This point will be still in $B(y, 4A|y|)$ by the choice of $A$.
% We can apply then the Boundary Harnack principle in the ball 
% $B(x^*, 2^{n-3}d(y))$ to the harmonic function $G(y,x)$ 
% to show that there exists a point $\widehat x$ such that 
% $d(\widehat x)>c_1 2^{n-3}d(y)$ and 
% $G(y,x)\le c_2 G(y,\widehat x)$ for some universal constants $c_1,c_2$. 
% We can now estimate 
We have, 
\begin{align*}
\frac{I_3(y)}{h(y)}&\le C
\int_{K \cap \left\{\delta(y)/2<|x-y|<\frac 12 |y|\right\}}
\frac{(\delta(x))^2}{h(x)|x-y|^{d}}\beta(x) dx\\
      &\le
C \int_{W_Z \cap \left\{\delta(y)/2<|x-y|<A|y|\right\}}
\frac{\gamma(\delta(x))}{|x-y|^d} dx. 
\end{align*}    
This is the same estimate as the starting estimate in 
~\cite[Lemma 14]{denisov_wachtel21}. 
The rest of the proof is the same. 
\end{proof}    

\begin{lemma}\label{lem:ghat.4}
There exists a bounded monotone decreasing function 
$\varepsilon_4(R)\to 0$, as $R\to\infty$,  such that for $y\in K: \delta(y)>R$,
\begin{equation}\label{eq:ghat.4}
I_4(y):=\int_{K \cap \{|x-y|<\delta(y)/2\}}\widehat G(x,y)\beta(x) dx \le \varepsilon_4(R) h(y).
\end{equation}
\end{lemma}    

\begin{proof}
We use the Harnack inequality to obtain that for $\lvert x-y \rvert < \delta(y)/2$
\begin{align*}
h(x) 
& \leq Ch(y).
\end{align*}
Therefore using this along with $\delta(x) \geq \delta(y)/2$
\begin{align*}
I_4(y)
& \leq \int_{K \cap \{|x-y|<\delta(y)/2\}} \frac{1}{\lvert y - x\rvert^{d-2}}
\frac{h(x) \gamma(\delta(x))}{\delta(x)^2} dx\\
& \leq C\int_{K \cap \{|x-y|<\delta(y)/2\}} \frac{1}{\lvert y - x\rvert^{d-2}}
\frac{h(y) \gamma(\delta(x))}{\delta(x)^2} dx\\
& \leq C
\int_{K \cap \{|x-y|<\delta(y)/2\}} \frac{1}{\lvert y - x\rvert^{d-2}}
\frac{4 h(y) \gamma(\delta(y)/2)}{\delta(y)^2} dx \\
& \leq C h(y) \gamma(\delta(y)/2).   \qedhere
\end{align*}

%The proof is the same as the proof of~\cite[Lemma 15]{denisov_wachtel21}.  
\end{proof}

\begin{lemma}\label{lem:bound.u.beta}
There exists a bounded monotone decreasing function 
$\varepsilon(R)\to 0$, as $R\to\infty$,  such that for $y\in K: \delta(y)>R$, then $U_\beta$ is finite, and the following estimate holds,
\begin{equation}\label{eq:bound.u.beta}
    U_\beta(y) < \varepsilon(R) h(y).
\end{equation}
\end{lemma}

\begin{proof}
This is proved by combining Lemma \ref{lem:ghat.1}, \ref{lem:ghat.2}, \ref{lem:ghat.3}
and \ref{lem:ghat.4}.
\end{proof}
%\begin{proof}
%The proof is the same as the proof of %Lemma~16~\cite{denisov_wachtel21}.
%\end{proof}
Next we will prove some estimates 
for the derivatives of $U_\beta$ that will be used 
in the estimates for the error of diffusion approximation.  
We will need the following notation: for $\theta\in (0,1]$ and an open $D$, 
 \[
    [f]_{\theta, D} := 
    \sup_{y,z\in D, y\neq z}  \frac{|f(y)-f(z)|}{|y-z|^\theta}. 
 \]
\begin{lemma}\label{lem:bound.u.beta.d}
    There exists a function $\varepsilon(R) \to \infty$ such that for $y \in K$ : $\delta(y) > R$ and $\theta \in (0,1]$, we have,
    \begin{equation}\label{eq:u.beta0}
        |U_\beta(y)| \le \varepsilon(R) h(y)
    \end{equation}
    \begin{equation}\label{eq:u.beta1}
        \left|\frac{\partial U_\beta(y)}{\partial y_i} \right| \le \varepsilon (R) \frac{h(y)}{\delta(y)}+c\delta(y) \beta(y)
    \end{equation}
    \begin{equation}\label{eq:u.beta2}
        \left|\frac{\partial^2 U_\beta(y)}{\partial y_i y_j} \right| \le \varepsilon (R) \frac{h(y)}{\delta(y)^2}+c \beta(y)
    \end{equation}
    \begin{equation}\label{eq:u.beta3}
        \left[(U_\beta)_{y_i y_j}\right]_{\theta, B\left(y,\frac{1}{3}\delta_Z(y)\right)} \le \varepsilon (R) \frac{h(y)}{\delta(y)^{2+\theta}}+c \frac{\beta(y)}{\delta(y)^\theta}
    \end{equation}
\end{lemma}
\begin{proof}
 The first inequality coincides with \eqref{eq:bound.u.beta}. 
Using the Theorem 4.6 in \cite{alma} with  the balls $B(y, r) $ and $B(y, 2r)$, where we take  $r=\frac{1}{3} \delta(y)$, we obtain 
\begin{equation}
    r(U_\beta(y))_{y_i} +r^2 (U_\beta(y))_{y_i y_j} \le C \left(\sup_{x\in B(y, 2r)} U_\beta(x) + r^{2+\theta} [\beta]_{\theta,B(y, 2r) }\right)
\end{equation}
and
\begin{equation}
    r^{2+\theta}[(U_\beta(y))_{y_i y_j}]_{\theta,B(y, r) } \le C \left( \sup_{x\in B(y, 2r)} U_\beta(x)+r^{2+\theta} [\beta]_{\theta,B(y, 2r) }\right). 
\end{equation}
In these equalities the first term can be estimated 
by using the Harnack inequality, 
\[
\sup_{x\in B(y, 2r)} U_\beta(x) 
\le \varepsilon(R) \sup_{x\in B(y, 2r)} h(x)
\le   \varepsilon(R) \frac{5/3}{(1/3)^{d-1}}h(y).
\]
We will  estimate 
$[\beta]_{\theta,B(y, 2r) }$.
For $x ,z \in B(y, \frac{2}{3}\delta_Z(y))$  and $\theta \in (0,1]$, by the triangle inequality, 
\[
\frac{|\beta(x)-\beta(z)|}{|x-z|^\theta} \le \frac{|h(x)-h(z)|}{|x-z|^\theta} \int_{\delta(x)}^{\infty}\frac{\gamma(t)}{t^3} dt + h(z) \frac{\left |\int_{\delta(z)}^{\delta(x)}\frac{\gamma(t)}{t^3} dt \right|}{|x-z|^\theta}. 
\]
Next note that 
\begin{equation}\label{eq.d.l}
|\delta(x)-\delta(z)| \le |x-z|.
\end{equation}
To prove this inequality consider first the case 
$\delta(x)>\delta(z)$. 
Let $\widetilde x$ be a point such that $|x-\widetilde x|=\delta(x)$ 
and note that $|z-\widetilde x|\ge\delta(z)$.
Then,  
\[
\delta(x)-\delta(z)
\le |x-\widetilde x|-|z-\widetilde x|
=|x-z+z-\widetilde x|-|z-\widetilde x|
\le |x-z|. 
\]
The case $\delta(x)<\delta(z)$ is symmetric. 

Applying the  mean value theorem 
and using the properties of slowly varying functions  
in the first line,
then using equation~\eqref{eq.d.l} and Karamata's  theorem in the second line we obtain 
\begin{align*}
    \frac{\left |\int_{\delta(z)}^{\delta(x)}\frac{\gamma(t)}{t^3} dt \right|}{|x-z|^\theta} &\le \frac{C\frac{\gamma(\delta(y))}{\delta(y)^3}\left|\delta(x)-\delta(z)\right|}{|x-z|^\theta} \\& \le C\frac{\gamma(\delta(y))}{\delta(y)^3} |\delta(x)-\delta(z)|^{1-\theta} \le \frac{C }{\delta(y)^\theta}\int_{\delta(y)}^{\infty}\frac{\gamma(t)}{t^3} dt. 
\end{align*}
Since $z\in B(y, \frac 23 \delta(y))$ we can make use of the Harnack inequality again to obtain,
\begin{equation}\label{eq:bound.h.z}
h^Z(z) \le 5\cdot 3^{d-2}h^Z(y). 
\end{equation}
Then,
\begin{equation}
    h(z) \frac{\left |\int_{\delta(z)}^{\delta(x)}\frac{\gamma(t)}{t^3} dt \right|}{|x-z|^\theta} \le c\frac{\beta(y)}{\delta(y)^\theta}. 
\end{equation}
Next, since $x, z\in B(y, \frac 23 \delta(y))$, 
applying Lemma~\ref{lem:harnack} we obtain 
\begin{align*}
   |h(x)-h(z)| &\le \left|\int^1_0 (\nabla  h(x+t(x-z)), x-z) dt\right| \\& \le |x-z| \max_{s\in B(y,\frac{2}{3} \delta(y))} \nabla h(s) \\& \le C |x-z| \max_{s\in B(y,\frac{2}{3} \delta(y))}  \frac{h(s)}{\delta(s)}
\end{align*}
for some finite constant $C$. 
It follows from~\eqref{eq:bound.h.z} that 
\[
\max_{s\in B(y,\frac{2}{3} \delta_Z(y))}  h(s) \le C_d h(y). 
\]
Since $s \in B\left(y, \frac{2}{3}\delta(y)\right)$, by the triangle inequality,
\[
\frac{1}{3} \delta(y) \le \delta(s) \le \frac{5}{3} \delta(y).
\]
Hence,
\[
|h(x)-h(z)|\le C|x-z| \frac{h(y)}{\delta(y)}.
\]
Since $x \in B\left(y, \frac{2}{3}\delta(y)\right)$, 
using the Karamata theorem, 
\begin{align*}
\frac{|h(x)-h(z)|}{|x-z|^\theta}\int_{\delta(x)}^{\infty}\frac{\gamma(t)}{t^3} dt &\le C \frac{|h(x)-h(z)}{|x-z|^\theta}\int_{\frac{1}{3}\delta(y)}^{\infty}\frac{\gamma(t)}{t^3} dt \\&  \le C\frac{|h(x)-h(z)|}{|x-z|^\theta} \frac{\gamma(\delta(y))}{\delta(y)^2} \\&\le C h(y)\frac{|x-z|^{1-\theta} }{\delta(y)} \int_{\delta(y)}^{\infty}\frac{\gamma(t)}{t^3} dt \\ 
&\le \frac{c \beta(y)}{\delta(y)^\theta}. \qedhere
\end{align*}
\end{proof}
Put $U_\beta(x)=0$ for $x\notin K$ and 
define the error of the  diffusion approximation 
\[
f_\beta(x):= \E\left[U_\beta(x+X); x+X\in K \right] -U_\beta(x),\quad x\in K. 
\]

\begin{lemma}
Suppose Assumption~\ref{ass-M2-h} holds. For any $\varepsilon >0$, there exists $R>0$ such that for $x\in K$ with $\delta(x)>R$, the following bound holds,
    \begin{equation}
    \left|f_\beta(x) + \frac{1}{2} \beta(x)\right| \le \varepsilon \beta(x). 
    \end{equation}
\end{lemma}
\begin{proof}
%We apply Taylor expansion on $U_\beta(x+X)$. Let $X$ be such that $ |\delta_Z(X)|<\frac{\delta_Z(x)}{2}
%$,
We first split the estimate to three parts,
%Let $m_Z(x)$ be defined as follows, 
%\begin{align*}
%    m_A(x)&:= \frac{1}{\sqrt 2}\max_{2\le j\le %}|x_j-x_{j-1}|;\\
%%    m_C(x)&:= \max(m_A(x),|x_1|);\\ 
%    m_D(x)&:= \max\left(m_A(x),\frac{|x_1+x_2|}%{\sqrt 2}\right). 
%\end{align*}
\[ 
\left|f_\beta(x)+\frac{1}{2}\beta(x)\right| \le 
E_1+E_2+E_3,
\]
where 
\begin{align*}
E_1&:=\E[U_\beta(x+X); |X|>\delta(x)/2],\\
E_2&:=U_\beta(x) \pr(|X|>\delta(x)/2),\\
E_3&:=\left|\E[U_\beta(x+X)-U_\beta(x); |X|\le \delta(x)/2]+\frac{1}{2}\beta(x)\right|. 
\end{align*}
We will apply the Taylor theorem to the third summand. 
By the Taylor theorem, on the event $\{|X|\le \delta(x)/2\}$, 
\begin{multline} 
    \left|U_\beta(x+X)-U_\beta(x)-\nabla U_\beta(x) \cdot X-\frac{1}{2}
    \sum_{1\le i, j\le d}
    \frac{\partial^2U(x)}{\partial x_i\partial x_j}X_iX_j\right| \\
    \le R_{2,\theta}(x) |X|^{2+\theta}, 
\end{multline}
  where 
      \[
          R_{2,\theta}(x)= 
          \sup_{i,j}
          \left[\frac{\partial^2U(x)}{\partial x_i\partial x_j}\right]_{\theta, B(x,\delta(x)/2)}<\infty. 
    \]
Then,  
\begin{multline*}
E_3\le 
\left|\E\left[\sum_{i=1}^d \frac{\partial U_\beta(x)}{\partial x_i} X_i; |X| \le \delta(x)/2\right]\right|\\
+
\left|\E\left[\frac{1}{2}\sum_{1\le i,j\le d}\frac{\partial^2 U_\beta(x)}{\partial x_i \partial x_j} X_i X_j; |X| \le \delta(x)/2\right] 
+\frac{1}{2}\beta(x) 
\right|
\\+ R_{2,\theta}(x)\E\left[|X|^{2+\theta};  |X|\le \delta(x)/2\right].
\end{multline*}
Making use of the assumption $\E[X]=0$, 
$\mathrm{cov}(X_i, X_j)=\delta_{i=j}$ and 
the definition 
of $U_\beta$ in~\eqref{eq:laplacian.drift} 
given by 
$\Delta U_\beta(x) =-\beta(x)$,  
we rearrange the terms to obtain 
\begin{multline*}
E_3\le 
\sum_{i=1}^d \left|\frac{\partial U_\beta(x)}{\partial x_i}\right|\E\left[|X|;|X|>\delta(x)/2\right]\\
+
\frac{1}{2}\sum_{1\le i,j\le d}\left|\frac{\partial^2 U_\beta(x)}{\partial x_i \partial x_j} \right|\E\left[|X|^2; |X|  >\delta(x)/2\right] \\ 
+ R_{2,\theta}(x)\E\left[|X|^{2+\theta}; |X|\le \delta(x)/2\right],  
\end{multline*}
where we also used that $|X_i| \le |X|$ and $|X_iX_j|\le |X|^2$. 
Applying further Lemma~\ref{lem:bound.u.beta.d}, 
\begin{multline*}
E_3\le 
\left(\varepsilon(R)\frac{h(x)}{\delta_Z(x)}+C\delta(x)\beta(x)\right)
d \E[|X|;|X|>\delta(x)/2]\\
+
\left(\varepsilon(R)\frac{h(x)}{(\delta(x))^2}+C\beta(x)\right)
\frac{d(d-1)}{2}
\E[|X|^2;|X|>\delta(x)/2]\\
+ \left(\varepsilon(R)\frac{h(x)}{\delta(x)^{2+\theta}}+C\frac{\beta(x)}{\delta(x)^\theta }\right)
\E\left[|X|^{2+\theta}; |X|\le \delta(x)/2\right]. 
\end{multline*}
 Since by Lemma~\ref{lem:construction.gamma}, 
\begin{align*}
\E[|X|;|X|>\delta(x)/2] &\le 
\frac{2}{\delta(x)}\E[|X|^2;|X|>\delta(x)/2]\\
\E[|X|^2;|X|>\delta(x)/2]
&= o(\gamma(\delta(x)))
\end{align*}
we obtain 
\begin{multline*}
E_3\le 
\left(\varepsilon(R)\frac{h(x)}{(\delta(x))^2}+\beta(x)\right)o(\gamma(\delta(x)))
\\
+ \left(\varepsilon(R)\frac{h(x)}{\delta(x)^{2+\theta}}+C\frac{\beta(x)}{\delta(x)^\theta }\right)\E\left[|X|^{2+\theta}; |X|<\delta(x)/2\right].
%\\
%\le 
%%\beta(x)\left(\widetilde \varepsilon(R) 
%+\frac{\E\left[|X|^{2+\theta}; |X|<\delta_Z(x)/2\right]}{\delta_Z(x)^\theta}
%\right). 
\end{multline*}
Next, using equation~\eqref{eq:gamma.integral.finite} and Karamata theorem,
\begin{multline*}
   \E\left[|X|^{2+\theta}, |X|\le \delta(x)/2\right] \le C\int^{\delta(x)/2}_{0}  y^{1+\theta} \P(|X|>y)dy  \\ \le C \int^{\delta(x)/2}_{0} y^{1+\theta} \frac{\E\left[|X|^{2};|X|>y\right]}{y^{2} } dy \\ 
   \le C\int^{\delta(x)/2}_{0}  y^{-1+\theta} o(\gamma(y)) dy \le C \delta_Z(x)^\theta o(\gamma(\delta(x))). 
  %\le \widetilde \varepsilon(R)\beta(x)
  %\delta_Z(x)^\theta \gamma(\delta_Z 
  %(x)). 
\end{multline*}
Hence, for some $\widetilde \varepsilon(R)\downarrow 0$,
\begin{equation}\label{e3}
E_3\le \widetilde \varepsilon(R)\beta(x). 
\end{equation}
To estimate the other two terms $E_1$ and $E_2$ observe the following inclusion 
\[
\{|X|>\delta(x)\}\subseteq
\bigcup_{i=1}^d\{|X_i|>\delta(x)/\sqrt{d}
\}. 
\]
Then using equation \eqref{eq:bound.u.beta}, we repeat the same argument expanding $h(x+X)$ by the Taylor formula in Lemma~\ref{lem:bound.f},
\begin{align}    
\label{e1}
E_1&\le \varepsilon(R) \E\left[ |h(x+X)|,|X|>\delta(x)/2 \right] \\
\nonumber &\le  
\varepsilon(R) \sum_{k=0}^{p} \frac{1}{k!} \sum_{\alpha_1+\cdots+\alpha_d =  k}\left|\frac{\partial^k h(x)}{\partial x_1^{\alpha_1} \cdots \partial x_d^{\alpha_d}}\right| 
\\
\nonumber 
&\phantom{xxx}\times \prod_{j\neq i}\E[|X_j|^{\alpha_j}]  \E[|X_i|^{\alpha_i}; |X_i|>\delta(x)/(\sqrt{d})]\\ 
\nonumber 
&\le c \varepsilon(R) \beta(x).
%\sum_{k=0}^p \sum_{m=0}^k \sum_{i=1}^d \E \left[ \frac{h^Z(x) |X_i|^m}{\delta_Z(x)^k},|X_i|>\delta_Z(x)/(2\sqrt{2}) \right]  \\& \le c \varepsilon(R) \sum_{k=0}^p \sum_{m=0}^k \sum_{i=1}^d h^Z(x) \frac{\gamma(\delta_Z(x))}{\delta_Z(x)^{k-m+2}} \le \varepsilon \beta(x)
\end{align}
Finally, by equation~\eqref{eq:gamma.integral.finite},
\begin{equation}
\label{e2}
  E_2\le \varepsilon(R) h(x) \sum^d_{i=1} \P\left(|X_i|>\delta(x)/(\sqrt{d}) \right) \le \varepsilon \beta(x).
\end{equation}
Combining~\eqref{e3},~\eqref{e1} and~\eqref{e2} we arrive the conclusion. 
\end{proof}
Next, we construct a positive supermartingale similarly to~\cite{denisov_wachtel21, denisov2023markov} as follows 
\[
V_\beta(x) = h(x+Rx_0)+ 3U_\beta(x+Rx_0).
\]
%where $x_0=(0,1,2,\ldots, d-1).$

\begin{lemma}\label{sum.beta}
   Suppose Assumption~\ref{ass-M2-h} holds. There exists a sufficiently  large $R>0$ such that $V_\beta(x)$ is a non-negative superharmonic function in $K$, that is $V_\beta(S(n)){\rm 1}\{\tau_x>n\}$ is a non-negative supermartingale. Moreover,
    \begin{equation}\label{eq:sum.beta}
    \E \left[\sum^{\tau_x-1}_{k=0} \beta(x+Rx_0 +S(k))\right] \le 2 V_\beta(x)  
\end{equation}
and  there exists a function $\varepsilon(R)\to 0$ such that 
\[
    \E \left[\sum^{\tau_x-1}_{k=1} |f (x+Rx_0+S(k))|\right] 
    \le     \varepsilon (R) h(x+Rx_0). 
\]
\end{lemma}
\begin{proof}
    The proof follows the same arguments as the proof of~\cite[Lemma 13]{denisov2023markov} 
    and hence is omitted. 
\end{proof}

\section{Construction of harmonic function}
\label{sec:harmonic}

%\subsection{Construction of harmonic function}
Recall $x_0$ is defined to satisfy the starlike property of the cone that $x_0 + K \subset K$ and $\text{dist}(x_0 + K, \partial K) > 0$. 
\begin{lemma}
Suppose Assumption~\ref{ass-M2-h} holds.
The function $V(x)$ defined in equation~\eqref{eq:harmonic} is positive, finite and harmonic for $(x+S(n))_{n \geq 0}$ killed at leaving $K$. 
It can be represented as follows 
for sufficiently large $R$ as 
%the harmonic is well-defined, positive, and finite for $\{x+S(n)\}$ killed at leaving $W_Z$,
\begin{multline}\label{harmonic.function}
    V(x)=h(x+Rx_0)- \E[h(x+Rx_0+S(\tau_x)); x+Rx_0+S(\tau_x)\in K] 
    \\
    +\E \left[\sum^{\tau_x-1}_{k=0} f (x+Rx_0+S(k))\right].  
\end{multline}
\end{lemma}
\begin{proof}
Put $h_+(x) = h(x)I(x\in K)$. 
    Consider a sequence  $(Z_n)_{n \ge 0}$ given by 
    \[
    Z_n := h_+(x+Rx_0+S(n \wedge \tau_x))-\sum^{n \wedge \tau_x-1} _{k=0} f(x+Rx_0+S(k)), \quad n\ge 0.
    \]
    This sequence is a martingale since 
\begin{multline*}
   \E[Z_n-Z_{n-1} |\mathcal{F}_{n-1}]\\= {\rm 1}\{\tau_x>n-1\} \E \left[h_+(x+Rx_0+S(n))-h_+(x+Rx_0+S(n-1))|\mathcal{F}_{n-1}\right]\\-{\rm 1}\{\tau_x>{n-1}\} f(x+Rx_0+S(n-1)) =0, 
\end{multline*}
where we have used the definition of $f$ 
in~\eqref{eq:defn.f}
given by 
$f(x)=\E[h_+(x+X)]-h(x)$ for $x\in K$. 
Applying the optional stopping theorem to $(Z_n)_{n\ge 0}$ and rearranging the terms,
\begin{equation}\label{eq:zn}
\begin{split}
  &  \E \left[h_+(x+Rx_0+S(n)), \tau_x>n\right] \\
  &=h(x+Rx_0)+\E\left[\sum^{n-1}_{k=0} f(x+Rx_0+S(k)), \tau_x>n\right]\\
  &\phantom{xx}- \E\left[h_+(x+Rx_0+S(\tau_x)), \tau_x\le n\right]\\
  &\phantom{xx}+\E\left[\sum^{\tau_x-1}_{k=0} f(x+Rx_0+S(k)), \tau_x\le n\right].   
  \end{split}
\end{equation}
Since the sequence $h_+(x+S(\tau_x)){\rm 1}\{\tau_x \le n\}$ is increasing along $n$, by the  monotone convergence theorem,
\[
 \lim_{n \to\infty}\E\left[h_+(x+Rx_0+S(\tau_x)), \tau_x<n\right] =\E\left[h_+(x+Rx_0+S(\tau_x))\right]. 
\]
The dominated convergence theorem together with Lemma \ref{sum.beta} 
implies that two limits below exist and finite 
\begin{align}
\nonumber
\lim_{n \to\infty} \E\left[\sum^{\tau_x-1}_{k=0} f(x+Rx_0+S(k)), \tau_x\le n\right] &= \E\left[\sum^{\tau_x-1}_{k=0} f(x+Rx_0+S(k))\right],\\ 
\label{dct}
\lim_{n \to\infty} \E\left[\sum^{n-1}_{k=0} f(x+Rx_0+S(k)), \tau_x>n\right] &=0. 
\end{align}
Hence,
\begin{multline*}
\lim_{n \to\infty} \E \left[h(x+Rx_0+S(n)), \tau_x>n\right] =h(x+Rx_0)
\\- \E\left[h_+(x+Rx_0+S(\tau_x)\right]+\E \left[\sum^{\tau_x-1}_{k=1} f (x+Rx_0+S(k))\right]. 
\end{multline*}
Since the left-hand side of the last equation is positive, expectation  
 $\E\left[h_+(x+Rx_0+S(\tau_x)\right]$ is finite and hence the right-hand side is finite. 

We are left to get rid of $Rx_0$. 
For that we will show that $\E \left[h(x+Rx_0+S(n)), \tau_x>n\right] \sim \E \left[h(x+S(n)), \tau_x>n\right]$ as $n\to\infty$.
For $x\in K$ we will makes use of the Taylor theorem again, 
\begin{multline*}
0\le h(x+Rx_0+S(n))
-h(x+S(n))\\
=
\sum_{k=1}^{p} \frac{1}{k!} \sum_{\alpha_1+\cdots+\alpha_d = k}\frac{\partial^k h(x+Rx_0+S(n))}{\partial x_1^{\alpha_1} \cdots \partial x_d^{\alpha_d}} 
(-R)^k
(x_0)_1^{\alpha_1} \cdots (x_0)_d^{\alpha_d}. 
\end{multline*}
On the event $\{\tau_x>n\}$, applying Lemma~\ref{lem:harnack} we see that 
\begin{multline*}
0\le h(x+Rx_0+S(n))
-h(x+S(n))\\
\le C 
\sum_{k=1}^{p} \frac{h(x+Rx_0+S(n))}{\delta(x+Rx_0+S(n))^k} R^k
\le CR 
\frac{h(x+Rx_0+S(n))}{\delta(x+Rx_0+S(n)}. 
\end{multline*}
Put
\[
\varepsilon_n :=\E\left[\beta(x+Rx_0+S(n))), \tau>n\right].  
\]
Notice that $\varepsilon_n \downarrow 0$ by Lemma~\ref{sum.beta}.
Put $a_n=1/\sqrt{\varepsilon_n}$ and note 
that $a_n\uparrow \infty$.  
Then, since $\frac{\gamma(x)}{x}$ is monotone decreasing, 
    \begin{align}
    \nonumber
        & \E\left[\frac{h(x+Rx_0+S(n))}{\delta(x+Rx_0+S(n))}, \tau_x>n,\delta(x+Rx_0+S(n)) \le a_n \right]  \\
        \label{eq:hrn.1}
        & \le \frac{a_n }{\gamma(a_n)} \E\left[\frac{h(x+Rx_0+S(n))}{\delta(x+Rx_0+S(n))}\frac{\gamma(\delta(x+Rx_0+S(n)))}{\delta(x+Rx_0+S(n))}, \tau_x>n\right] \nonumber \\ 
        \nonumber
        & \le C \frac{a_n }{\gamma(a_n)} \E\left[\beta(x+Rx_0+S(n)), \tau_x>n\right] \nonumber \\
        &\le C\frac{a_n}{\gamma(a_n)} \varepsilon_n 
        = C\frac{1}{a_n \gamma(a_n)}
        \rightarrow 0,
    \end{align}
since $\gamma$ is a slowly varying function.  
Next use that $h(x+Rx_0) \leq V_{\beta}(x)$ and that $V_{\beta}$ is a supermartingale
\begin{equation}\label{eq:hrn.2}
    \begin{split}
        & \E\left[\frac{h(x+Rx_0+S(n))}{\delta(x+Rx_0+S(n))}, \tau_x>n,\delta(x+Rx_0+S(n)) > a_n \right] \\& \le \frac{C}{a_n} \E\left[h(x+Rx_0+S(n)), \tau_x>n\right] \\
       & \le \frac{C}{a_n} \E\left[V_{\beta}(x+S(n)), \tau_x>n\right]\\
       & \le \frac{C}{a_n}V_\beta(x) \rightarrow 0.
    \end{split}
\end{equation}
Combining  equation~\eqref{eq:hrn.1} and \eqref{eq:hrn.2} we obtain, 
\[
\left|\E \left[h(x+Rx_0+S(n)), \tau_x>n\right] - \E \left[h(x+S(n)), \tau_x>n\right] \right| \rightarrow 0
\]
as $n \rightarrow \infty$.

Therefore, the  function $V(x)$ below is well defined 
and does not depend on $R$ where the right hand side is valid for sufficiently large $R$,  
\begin{align}\label{eq:harmonic.function}
   & V(x)=\lim_{n \to\infty} \E \left[h(x+S(n)), \tau_x>n\right] \\&=h(x+Rx_0)- \E\left[h(x+Rx_0+S(\tau_x)\right]+\E \left[\sum^{\tau_x-1}_{k=1} f (x+Rx_0+S(k))\right]. \nonumber
\end{align}
By Lemma \ref{sum.beta},
\[
V(x) \le C h(x+Rx_0). 
\]
Then, the harmonic property of $V(x)$ follows from the the sequence of inequalities below together with dominated convergence theorem  
\begin{align*}
V(x)&=\lim_{n \to\infty}
\int_K \P(x+S(1)\in dy)
\E \left[h(y+S(n+1)), \tau_y>n-1\right]\\
&=
\int_K \P(x+S(1)\in dy)
\lim_{n \to\infty}
\E \left[h(y+S(n+1)), \tau_y>n-1\right]\\
&=\int_K \P(x+S(1)\in dy) V(y)
=\E[V(x+S(1));\tau_x>1]. \qedhere
\end{align*}
\end{proof}

\section{Upper bound  for $\pr(\tau_x>n)$}
\label{sec:tail_asy_bound}

We need first two auxiliary estimate. 
Let $\xi$ be a non-negative random variable. 
Then, for any $A>0$, 
\begin{equation}\label{eq:lower.bound.xi}
\E[\xi]\ge 
\frac{1}{A} \left(\E[\xi^2]-\E[\xi^2;\xi>A]\right). 
\end{equation}
Indeed, 
using that for non-negative $\xi$,
\begin{equation*}
\E\left[\frac{\xi^2}{A}; \xi \leq A\right] \leq \E[\xi; \xi \leq A] \leq \E[\xi]
\end{equation*}
we obtain 
\[
\E[\xi] \ge \frac{1}{A}\left(\E[\xi^2]-\E[\xi^2;\xi>A]\right). 
\]

To obtain a sharp upper bound of $\P(\tau_x>n)$, we will start by estimating it when the starting point near boundary. 
\begin{lemma}\label{lemma:bound.near.boundary}
Let Assumption~\ref{ass-m-h} hold. 
Then there exists a constant C such that,
\[
\P\left(\tau_x>n \right) \le C \frac{\delta(x+x_0)}{\sqrt{n}}. 
\]
In particular, there exists $C$ and $q>0$ such that 
for any $\varepsilon>0$, 
\[
    \sup_{x\in W_Z\colon \delta(x)< \varepsilon \sqrt{n}} \P(\tau_x >n) < C\varepsilon^q,\quad n\ge 1/\varepsilon^{2q}. 
\]
\end{lemma}
\begin{remark}
For Lipschitz cones this Lemma was proved in~\cite[Lemma 11]{denisov2023markov}. Lemma 11~\cite[Lemma 11]{denisov2023markov} has a stronger requirement 
that~\eqref{eq:harmonic.boundary} and~\eqref{eq:harmonic.boundary.lower} hold, but they were not used and the Lipschitz property is sufficient. 
\end{remark}
\begin{proof}
We give a  proof in the  simpler case when $K$ is convex,  see~\cite[Lemma 11]{denisov2023markov} for the general case. In the convex case $q=1$. 
For any unit vector $\textbf{e}$ 
such that $(x,\textbf{e})>0$ 
consider a one-dimensional stopping time $\tau_x^{(\textbf{e})}$,
    \[
    \tau_x^{(\textbf{e})}:= \{k\ge 1, (x+S(k), \textbf{e}) \le 0\}. 
    \]
For a fixed vector $\textbf{e}$  the required estimate 
follows immediately from (30) in \cite[Lemma 3]{denisov_wachtel2016}.     
However, we need a  derive a uniform in $\textbf{e}$ bound for $\P(\tau_x^{(\textbf{e})}>n)$ which require more work. 

Denote $S^{\textbf e} (n) = (S(n),\textbf e)$ and 
$X^{\textbf e}_i(j) := (X_i(j),\textbf e)$.

It follows from~\cite[Lemma~25]{denisov_sakhanenko_wachtel18} that 
\[
\pr(\tau_x^{(\textbf{e})}>n)
\le \frac{\E[(x+S(n),e);\tau_x^{(\textbf{e})}>n]}{\E[(S(n),e)^+]}. 
\]

We will estimate $\E[(S^{\textbf e} (n))^+]$  first. 
Since $\E[S^{\textbf e} (n)]=0$ we obtain 
that $\E[(S^{\textbf e}(n))^+] = \frac12\E[|S^{\textbf e} (n)|]$. 
Fix $A>0$ that will be picked later. Then, using inequality~\eqref{eq:lower.bound.xi} 
\begin{align*}
\E[|S^{\textbf e} (n)|]
&=A\sqrt n 
\E\left[\frac{|(S(n),e)|}{A\sqrt n}\right]\\
&\ge 
\frac{\sqrt{n}}{A}
\left(
1- 
\E\left[\left(\frac{|S(n)|}{\sqrt n}\right)^2;|S(n)|> A\sqrt n\right]
\right),
\end{align*} 
where we used that $|S^{\textbf e}(n)|\le |S(n)|$
and $\E[(S^{\textbf e} (n))^2]=1$.  
By the Central Limit Theorem $\frac{S(n)}{\sqrt n}$ converges 
to the multivariate standard normal distribution. 
Using this convergence one can show that 
the sequence 
$\left(\frac{S(n)}{\sqrt n}\right)^2$ 
is uniformly integrable and hence there exists $A$ such that 
\[
\E\left[\left(\frac{|S_n|}{\sqrt n}\right)^2;|S_n|> A\sqrt n\right]
\le \frac12.
\]
For this value of $A$ we obtain that 
uniformly in $\mathbf e$
\[
\E[(S(n),e)^+]
\ge \frac{\sqrt n}{4A}. 
\]

Next we obtain  an upper bound for $\E[(x+S_n,e)$. 
First note that 
\[
\E[(x+S(n),e); \tau^{\textbf e}_x>n] 
\le x^{\textbf e} - 
\E[ S^{\textbf e}(\tau^{\textbf e})].
\]
Put $\sigma_-:=\inf\{n\ge 1\colon S^{\textbf e}(e) \le 0 \}$ and $\sigma_+:=\inf\{n\ge 1\colon S^{\textbf e}(e) > 0 \}$. Let $H_-(x)$ be the renewal function of the decreasing ladder height process of $(S^{\textbf e}(n))_{n\ge 0}$. 
Then, it follows from  Wald's identity that 
\[
x^{\textbf e} - 
\E[ S^{\textbf e}(\tau^{\textbf e})] 
= \E[S^{\textbf e}(\sigma_-)] H_-(x^{\textbf e}). 
\]
By the Erickson bound for the renewal function (see Lemma~1 in~\cite{Erickson73}), 
\[
H_-(x^{\textbf e}) \le 2 \frac{x^{\textbf e}}{\int_0^x \pr(|S^{\textbf e}(\sigma_-)|>t)dt}
\le 
2 \frac{x^{\textbf e}}{\int_0^{x^{\textbf e}} \pr((X^{\textbf e})^->t)dt}.
\]
By the Paley-Ziegmund inequality, 
\[
\pr\left((X^{\textbf e})^->\frac{1}{2} \E[(X^{\textbf e})^-]\right)
\ge \frac{1}{4} \frac{(\E[(X^{\textbf e})^-])^2}{(\E[(X^{\textbf e})^-)^2]}.
\]
We can further estimate 
$\E[(X^{\textbf e})^-)^2\le \E[(X^{\textbf e}))^2=1$ and 
$\E[(X^{\textbf e})^-] = \frac 12\E[|X^{\textbf e}| $. 
Now fix $A_0$ and estimate further 
using~\eqref{eq:lower.bound.xi}
\[
\E[|X^{\textbf e}|] 
\ge\frac{1}{A_0} \left(1- \E[|X^{\textbf e}|^2; |X^{\textbf e}|>A_0]\right)
\ge 
\frac{1}{A_0} \left(1- \E[|X|^2; |X|> A_0]\right).
\] 
Now pick $A_0$ sufficiently large 
to ensure that $\E[|X|^2; |X|> A_0]\le \frac 12$ to obtain 
\begin{equation}\label{eq:pz}
\E[(X^{\textbf e})^-]  \ge \frac{1}{4A_0}.
\end{equation}
Function $\frac{r}{\int_0^r \pr((X^{\textbf e})^->t)dt}$ 
is  non-decreasing and, therefore, 
\begin{align*}
\frac{x^{\textbf e}}{\int_0^{x^{\textbf e}} \pr((X^{\textbf e})^->t)dt} 
&\le 
\frac{x^{\textbf e}+\frac 12 \E[(X^{\textbf e})^-]}{\int_0^{x+\frac 12 \E[(X^{\textbf e})^-]} \pr((X^{\textbf e})^->t)dt} \\
&\le 
\frac{x^{\textbf e}+\frac 12\E[(X^{\textbf e})^-]}{\int_0^{\frac 12\E[(X^{\textbf e})^-]} \pr((X^{\textbf e})^->t)dt} \\
&\le 
\frac{x^{\textbf e}+\frac 12\E[(X^{\textbf e})^-]}{\frac 12\E[(X^{\textbf e})^-] \pr((X^{\textbf e})^->\frac 12\E[(X^{\textbf e})^-])}. 
\end{align*}
We can now makes use of the estimate 
$\E[(X^{\textbf e})^-] \le \sqrt {\E[((X^{\textbf e})^-)^2]}\le 1$ and~\eqref{eq:pz} to obtain 
\[
\frac{x^{\textbf e}}{\int_0^{x^{\textbf e}} \pr((X^{\textbf e})^->t)dt}  
\le \frac{x^{\textbf e}+\frac{1}{2}}{\frac{1}{8A_0}}
=8A_0\left(x^{\textbf e}+\frac{1}{2}\right).
\]
Next we can note that 
\[
\E[S^{\textbf e}(\sigma_-)]
=\frac{1}{2 \E[S^{\textbf e}(\sigma_+)]}
\le 
\frac{1}{2 E [(X^{\textbf e}])^- }
\le 2A_0. 
\]
Combining these estimates we obtain 
\[
x^{\textbf e} - 
\E[ S^{\textbf e}(\tau^{\textbf e}_x)] 
\le 32A_0^2 \left(x^{\textbf e}+\frac{1}{2}\right)
\]
As a result, 
\[
\pr(\tau_x^{(\textbf{e})}>n) \le 
64A_0^2A
\frac{x^{\textbf e}+\frac{1}{2}}{\sqrt n},
\]
where  constants $A$ and $A_0$ are independent of direction $\textbf{e}.$ 
Since this inequality holds for each $\textbf{e}$ we can see that 
    \[
    \P(\tau_x>n) \le C 
    \frac{\delta(x+x_0)}{\sqrt{n}}.
    \]
    Hence, the first statement follows. 
    Now let $\delta(x)<\varepsilon \sqrt{n}$. 
    Taking $n\ge 1/\varepsilon^2$, 
    \[
    \P(\tau_x>n)\le C \frac{1+\varepsilon \sqrt{n}}{n}
    \le 2C\varepsilon
    \]
    we obtain the second statement. 
\end{proof}
\begin{lemma}\label{lem:p.tau}
Let  Assumption \ref{ass-m-h}  with $p\ge 1$ hold. Then there exists a constant $C$ such that 
\[
\P(\tau_x>n) \le C\frac{V_\beta(x)}{n^{p/2}} \le C\frac{h(x+Rx_0)}{n^{p/2}}.
\]
\end{lemma}
\begin{proof}
We fix $\varepsilon>0$, 
then split the probability in two parts,
\begin{align*}
\P(\tau_x>n) &\le \P\left(\tau_x>n, \delta(x+S(n/2)) \ge \varepsilon \sqrt{n}\right)\\&+ \P\left(\tau_x>n, \delta(x+S(n/2)) <\varepsilon \sqrt{n}\right). 
\end{align*}
Applying  the  Markov inequality and using the fact that  
$V_\beta(x+S(n)) \mathbbm{1}\{\tau_x>n\}$ is a supermartingale  we obtain,
using $h(x) \geq \delta(x)^p$, 
\begin{equation}\label{eq:p.tau.1}
\begin{split}
    &\P\left(\tau_x>n, \delta(x+S(n/2)) \ge \varepsilon \sqrt{n}\right)\le C\frac{\E[h(x+S(n/2)), \tau_x>n]}{(\varepsilon\sqrt{n})^p}\\& \le C\frac{\E[V_\beta (x+S(n/2)), \tau_x>n/2]}{(\varepsilon\sqrt{n})^p} \le C\frac{V_\beta(x)}{(\varepsilon\sqrt{n})^p} \le \frac{C}{\varepsilon^{p/2}}\frac{h(x+Rx_0)}{n^{p/2}}, 
    \end{split}
\end{equation}
where we also used Lemma~\ref{lem:bound.u.beta} in the last inequality. 
Using the  Markov property in the first inequality 
and Lemma~\ref{lemma:bound.near.boundary} in the second we obtain 
\begin{multline}\label{eq:p.tau.2}
\P\left(\tau_x>n, \delta(x+S(n/2)) <\varepsilon \sqrt{n}\right)\\\le
    \int_{\delta(y)< \varepsilon \sqrt{n}} \P\left(\tau_x >\frac{n}{2}, x+S\left(\frac{n}{2}\right) \in dy\right) \P\left(\tau_y>\frac{n}{2}\right)\\
\le C\varepsilon^q \P(\tau_x>n/2). 
\end{multline}
Combining~\eqref{eq:p.tau.1} and \eqref{eq:p.tau.2} together,
\[
\P(\tau_x>n) \le C\frac{V_\beta(x)}{(\varepsilon \sqrt{n})^p}+ C\varepsilon^q \P\left(\tau_x> \frac{n}{2}\right),
\]
where constant $C$ does not depend on $n$ and $\varepsilon$. 
Iterating it $N$ times,
\[
\P(\tau_x>n) \le C \sum^{N-1}_{i=0} (C\varepsilon^q)^i \frac{V_\beta(x)}{(\varepsilon \sqrt{n})^p}+ \varepsilon^{qN} \P\left(\tau_x> \frac{n}{2^N}\right).
\]
Now we can take $\varepsilon $ sufficiently small 
to ensure that $C\varepsilon^q<1$. 
Then
\[
\P(\tau_x>n) \le \frac{C}{1-C\varepsilon^q} \frac{1}{\varepsilon^{p/2}}\frac{V_\beta(x)}{{n}^{p/2}}+ (C\varepsilon^q)^N. 
\]
Taking $N> \frac{p\log n}{-2\ln (C\varepsilon^q)}$ 
we then obtain 
\[
\P(\tau_x>n) \le \frac{C}{1-C\varepsilon} \frac{1}{\varepsilon^{p/2}}\frac{V_\beta(x)}{{n}^{p/2}}+ \frac{1}{n^{p/2}}. 
\]
The statement then follows. 
\end{proof}

\section{Proof of Theorem~\ref{thm.p.tau-poly}.}
\label{sec:tail_asymptotics}

\begin{proof}[Proof of Theorem~\ref{thm.p.tau-poly}] 
Choose $m=m(n)=o(n)$ and fix $\varepsilon<1$ and $\lambda>1$. 
Let 
\begin{align*} 
M(x)  &= \max(|x_1|,\ldots, |x_d|), \quad x=(x_1,\ldots,x_d)\in K.
\end{align*}
Then we divide $K$ into three domains, \begin{align*}
    K_{1} &:=\left\{y \in K, \delta(y) \le \varepsilon \sqrt{m}, M(y)\le \lambda\sqrt{m}\right\} \\ 
    K_{2} &:=\left\{y \in K, \delta(y) > \varepsilon \sqrt{m}, M(y) \le \lambda \sqrt{m}\right\} \\ K_{3} &:=\left\{y \in K, M(y) > \lambda \sqrt{m}\right\}. 
\end{align*}
Let $D$ be a compact set in $K$ or $D=K$.   Using Markov property at $m <n$, 
\[
\P \left(\frac{x+S(n)}{\sqrt{n}} \in D, \tau_x>n\right) = P_1+P_2+P_3,
\]
where 
    \begin{align*}
      P_1&:= \int_{K_{1}} \P(x+S(m) \in d y, \tau_x>m) \P \left(\frac{y+S(n-m)}{\sqrt{n}} \in D, \tau_y > n-m\right) \\
      P_2&:=\int_{K_{2}} \P(x+S(m) \in d y, \tau_x>m) \P \left(\frac{y+S(n-m)}{\sqrt{n}} \in D, \tau_y > n-m\right)\\P_3&:=\int_{K_{3}} \P(x+S(m) \in d y, \tau_x>m) \P \left(\frac{y+S(n-m)}{\sqrt{n}} \in D, \tau_y > n-m\right).
    \end{align*}
The first integral can be bounded using  Lemma~\ref{lem:p.tau},
\begin{align*}
P_1& \le \int_{K_{1}} \P(x+S(m) \in d y, \tau_x>m) \P \left(\tau_y > n-m\right) \\& \le \frac{C}{n^{p/2}} \int_{K_{1}} \P(x+S(m) \in d y, \tau_x>m) h(y+Rx_0)
\end{align*}
Since
\[
h(y+Rx_0) \le \varepsilon\sqrt{m} (\lambda \sqrt{m})^{p-1}\le \varepsilon \lambda^{p-1} {m}^{p/2},\quad y\in W_{K},
\]
 using  Lemma~\ref{lem:p.tau} again,
\begin{equation}
    \label{eq:bound.w1}
    P_1\le \frac{C}{n^{p/2}} \varepsilon \lambda^{p-1} \sqrt{m}^{p}\P_x \left(\tau_x > m\right) \le \frac{C}{n^{p/2}} \varepsilon \lambda^{p-1} h(x+Rx_0). 
\end{equation}
Next we consider $P_3$.  Applying  Lemma~\ref{lem:p.tau},
\begin{align*}
   P_3& \le \frac{C}{n^{p/2}} \int_{K_{3}} \P(x+S(m) \in d y, \tau_x>m) h(y+Rx_0) \\& \le \frac{C}{n^{p/2}} \E \left[h(x+Rx_0+S(m)), \tau_x >m, x+S(m) \in K_{3}\right]. 
\end{align*}
By the Taylor expansion 
\[
h(x+Rx_0+S(m)) 
\le 
\sum_{k=0}^{p} \frac{1}{k!} \sum_{\alpha_1+\cdots+\alpha_d = k}\frac{\partial^k h(x+Rx_0)}{\partial x_1^{\alpha_1} \cdots \partial x_d^{\alpha_d}} 
|S_1(m)|^{\alpha_1} \cdots |S_d(m)|^{\alpha_d}. 
\]
%where $\mathcal{I}=\left \{(\alpha_0, \alpha_1,\cdots,\alpha_d), \quad \sum_{i=0}^{d} \alpha_i=p,\quad  0\le \alpha_i\le r_Z\right\}$, and $C_{R}(x)$ 
%is a polynomial in  $x$ of  degree $\alpha_0$,  whose coefficients depend only on $R$, $x_0$ and some absolute constants.
Then we get the following bound, uniformly in $M(x)= o(\sqrt{n})$, 
\begin{align*}
    &\E \left[h(x+Rx_0+S(m)), \tau_x >m, x+S(m) \in K_{3}\right] \\& \le
    \sum_{k=0}^{p} \frac{1}{k!} \sum_{\alpha_1+\cdots+\alpha_d = k}\frac{\partial^k h(x+Rx_0)}{\partial x_1^{\alpha_1} \cdots \partial x_d^{\alpha_d}} \E \left[ |S_1(m)|^{\alpha_1} \cdots|S_d(m)|^{\alpha_d}, x+S(m)\in K_{3}\right]
\\
    & 
    \le \sum_{k=0}^{p} \frac{1}{k!} \sum_{\alpha_1+\cdots+\alpha_d = k}\frac{\partial^k h(x+Rx_0)}{\partial x_1^{\alpha_1} \cdots \partial x_d^{\alpha_d}}
    \sum_{j=1}^d 
    \E \left[|S_j(m)|^{\alpha_{j}}, |S_j(m)|>\lambda \sqrt{m}\right].
\end{align*}
% \textcolor{blue}{(now that is a one-dimensional problem is there something we can cite?)}
Making  use of the equation (53) in \cite{denisov2023markov},
\[
\E \left[|S_{j}(m)|^{\alpha_{j}}, |S_{j}(m)|>\frac{1}{2}\lambda \sqrt{m}\right] \le C \int_{\lambda \sqrt{m}/2}^\infty z^{\alpha_{j}-1} \P(|S_{j}(m)|>z) dz.
\]
From Corollary 23 in \cite{denisov_wachtel15} with $x=z$ and $y=\delta z$, we obtain,
\[
\P(|S_{j}(m)|>z) \le 2d e^{\frac{1}{\delta \sqrt{n}}}\left(\frac{\sqrt{d} m}{\delta z^2}\right)^{\frac{1}{\delta \sqrt{n}}} +m \P\left(|X_j(1)|>\delta z\right). 
\]
Integrating the first integral by fixing sufficiently  small $\varepsilon$ such that $\alpha_{j}-1-\frac{2}{\delta\sqrt{d}} \le -2$,
\begin{equation}\label{eq:fuk.0}
    \begin{split}
    &  C\int_{\lambda\sqrt{m}/2}^\infty z^{\alpha_{j}-1} 2d e^{\frac{1}{\delta \sqrt{d}}}\left(\frac{\sqrt{d} m}{\delta  z^2}\right)^{\frac{1}{\delta \sqrt{d}}} dz \\& \le 2Cd e^{\frac{1}{\delta \sqrt{d}}}\left(\frac{\sqrt{d} m}{\delta }\right)^{\frac{1}{\delta \sqrt{d}}} \left(\frac{\lambda\sqrt{m}}{2}\right)^{\alpha_{j}-\frac{2}{\delta\sqrt{d}}}\\&\le 2Cd e^{\frac{1}{\varepsilon \sqrt{d}}}\left(\frac{\sqrt{d}}{\varepsilon }\right)^{\frac{1}{\delta \sqrt{d}}} \left(\frac{\lambda}{2}\right)^{\alpha_{j}-\frac{2}{\delta\sqrt{d}}} m^{\alpha_{j}/2} \le C m^{\alpha_{j}/2} 
    \lambda^{-1}. 
 \end{split}
\end{equation}
We estimate the second integral separately for $r=2$  and other cases. For $r=2$, using  Assumption \ref{ass-m}, the Markov inequality and equation (53) in \cite{denisov2023markov},
\begin{equation}\label{eq:fuk.1}
    \begin{split}
    & C \int_{\lambda \sqrt{m}/2}^\infty z^{\alpha_{j}-1} m \P\left(|X_{j}(1)|>\delta z\right) dz \\&\le \frac{Cm}{\delta^{\alpha_{j}} }\int_{\lambda\sqrt{m} \delta/2}^\infty z^{\alpha_{j}-1}\P\left(|X_{j}(1)|> z\right) dz \\&\le \frac{Cm}{\delta^{\alpha_{j}} } \left(\frac{2}{\lambda\sqrt{m}\delta}\right)^{2-\alpha_{j}} \int_{\lambda\sqrt{m} \delta/2}^\infty z\P\left(|X_{j}(1)|> z\right) dz \\&\le \frac{C m^{\alpha_{j}/2}}{\delta^2} \frac{\E\left[|X|^2\log{1+|X|}, |X|>\frac{\lambda\sqrt{m}\delta}{2}\right]}{\log{(1+\lambda\sqrt{m}\delta/2})} := \frac{Cm^{\alpha_{j}/2}}{\delta^2} f_1(\lambda),
   \end{split}
\end{equation}
where $f_1(\lambda)\rightarrow 0$ as $\lambda\rightarrow \infty$.

For $r>2$ in other cases, using Assumption \ref{ass-m} and Markov inequality again, 
\begin{equation}\label{eq:fuk.2}
    \begin{split}
     &  C\int_{\lambda\sqrt{m}/2}^\infty z^{\alpha_{j}-1} m \P\left(|X_{j}(1)|>\delta z\right) dz \\&\le \frac{Cm}{\delta^{\alpha_{j}} }\int_{\lambda\sqrt{m} \delta/2}^\infty z^{\alpha_{j}-1}\P\left(|X_{j}(1)|> z\right) dz \\&\le \frac{Cm}{\delta^{\alpha_{j}}} \frac{ \E\left[|X|^{r}, |X|>\frac{\lambda\sqrt{m}\delta}{2}\right]}{(\lambda\sqrt{m}\delta)^{r-\alpha_{j}}}\\&\le \frac{c \sqrt{m}^{2-r+\alpha_{j}}}{\lambda^{r-\alpha_{j}}}\E\left[|X|^{r}, |X|>\frac{\lambda\sqrt{m}\delta}{2}\right] \\&\le  \frac{c \sqrt{m}^{\alpha_{j}}}{\lambda^{r_Z-\alpha_{j}}}\E\left[|X|^{r}, |X|>\frac{\lambda\sqrt{m}\delta}{2}\right] := c m^{\alpha_{j}/2} f_2(\lambda)    
    \end{split}
\end{equation}
Since $\alpha_{j}<r$, we can see that $f_2(\lambda) \rightarrow 0$ as $\lambda \rightarrow \infty$.
Combining equation~\eqref{eq:fuk.0}, \eqref{eq:fuk.1} and \eqref{eq:fuk.2}, we can conclude that,
\begin{equation}  \label{eq:fuk}
\E \left[|S_{j}(m)|^{\alpha_{j}}, |S_{j}(m)|>\lambda\sqrt{m}\right] \le C \sqrt{m}^{\alpha_{j}} f(\lambda),
\end{equation}
where $f(\lambda) \rightarrow 0$ as $\lambda \rightarrow \infty$.
With equation~\eqref{eq:fuk}, integral over $K_{3}$ can be bounded as,
\begin{align}
    \label{eq:bound.w3}
   & \frac{C}{n^{p/2}}
   \sum_{k=0}^{p} \frac{1}{k!} \sum_{\alpha_1+\cdots+\alpha_d = k}\frac{\partial^k h(x+Rx_0)}{\partial x_1^{\alpha_1} \cdots \partial x_d^{\alpha_d}}
   \sum_{j=1}^{d} (\lambda\sqrt{m})^{p-\alpha_{j}}  \E \left[|S_{j}(m)|^{\alpha_{j}}, |S_j(m)|>\frac{1}{2}\lambda\sqrt{m}\right] \nonumber \\& \le \frac{C}{n^{p/2}}
   \sum_{k=0}^{p} \frac{1}{k!} \sum_{\alpha_1+\cdots+\alpha_d = k}\frac{h(x+Rx_0)}{\delta(x+Rx_0)^k}
   \sum_{j=1}^{d} (\lambda\sqrt{m})^{p-\alpha_{j}} m^{\alpha_{j}/2} f(\lambda) \nonumber \\
   & \le Ch(x+Rx_0)f(\lambda),
   \end{align}
where $f(\lambda) \rightarrow 0$ as $\lambda \rightarrow \infty$.

Next, 
using a  strong coupling of random walks with Brownian motion 
similarly to~\cite{denisov_wachtel21} 
% applying  the Donsker Functional Central Limit Theorem  
we obtain  uniformly in $y \in K_{2}$ with sufficiently large $m(n)=o(n)$, 
\begin{align*}
    \P\left(\frac{y+S(n-m)}{\sqrt{n}} \in D, \tau_y>n-m\right) & 
    \sim \P\left(\frac{y}{\sqrt{n}}+B\left(\frac{n-m}{n}\right)\in D, \tau^{bm}_y >n-m\right) \\& \sim \varkappa\frac{h(y)}{n^{\frac{p}{2}}} \int_{D}\exp\left({-\frac{|z|^2}{2}}\right) h(z) d z. 
\end{align*}
Using Lemma~\ref{lem:p.tau}, 
\begin{align*}
   P_2& \sim \frac{\varkappa}{n^{\frac{p}{2}}} \int_{D} e^{-|z|^2/2} h^Z(z) dz\int_{K_{2}} \P(x+S(m) \in d y, \tau_x>m) h(y).  
\end{align*}
Estimating  the integral over $K_{2}$ by the bounds of $K_{1}$ and $K_{3}$,
\begin{equation}
\label{eq:bound.w2}
    \begin{split}
    &\Biggl|\int_{K_{2}} \P\left(x+S(m) \in d y, \tau_x>m\right) \P\left(\frac{y+S(m)}{\sqrt{n}} \in D, \tau_y > n-m\right) \\& - \frac{C}{n^{\frac{p}{2}}} \int_D h(z) \exp\left({-\frac{|z|^2}{2}}\right) d z \int_{K} \P(x+S(m)\in dy, \tau_x>m)h(y) \Biggr| \\& =\left|\frac{C}{n^{\frac{p}{2}}}\int_{K_{1} \bigcup K_{3}}\P(x+S(m) \in d y, \tau_x>m) h^Z(y)\ \int_{D} h^Z(z) \exp\left({-\frac{|z|^2}{2}}\right) d z\right| \\& \le \frac{C}{n^{\frac{p}{2}}} \left(\varepsilon \lambda^{p-1}\right) h(x+Rx_0)+Cf(\lambda).
    \end{split}
\end{equation}
Combine these three bounds of equation \eqref{eq:bound.w1}, \eqref{eq:bound.w2} and \eqref{eq:bound.w3},
\begin{align*}
&\Bigg|\P\left(\frac{x+S(n)}{\sqrt{n}}\in D,\tau_x>n\right)
-\frac{\varkappa}{n^{p/2}}\int_D h(z)e^{-|z|^2/2}dz
\E[h(x+S(m));\tau_x>m]\Bigg|\\
&\hspace{1cm} \le \frac{C}{n^{\frac{p}{2}}} \left(\varepsilon \lambda^{p-1}\right) h(x+Rx_0)+Cf(\lambda).
\end{align*}
uniformly in $\max_{1\le j \le d}\{|x_j|\}< \sqrt{m}$.
Letting  $\varepsilon \to 0$ and $\lambda \to \infty$, and taking  the limit with $m \rightarrow \infty$, we prove part (ii) of  Theorem \ref{thm.p.tau-poly} since
\begin{equation}
    \P \left( \frac{x+S(n)}{\sqrt{n}} \in D, \tau_x >n \right) \sim \frac{\varkappa V(x)}{n^\frac{p}{2}} \int_D h(z)\exp\left\{-\frac{|z|^2}{2}\right\} d z
\end{equation}
and 
\begin{equation*}
    \P(\tau_x>n) \sim \frac{\varkappa V(x)}{n^\frac{p}{2}}
    \int_K h(z)\exp\left\{-\frac{|z|^2}{2}\right\} d z
    . 
\end{equation*}
\qedhere

\end{proof}

\appendix

\bibliographystyle{abbrv}
 \bibliography{ExpOrdWalk}

\end{document}